\documentclass[11pt]{amsart}
\usepackage{amssymb}
\usepackage{bm}
\usepackage{graphicx}
\usepackage[centertags]{amsmath}
\usepackage{amsfonts}
\usepackage{amsthm}
\usepackage{a4}
\usepackage{latexsym}
\usepackage{amsmath}
\usepackage{amsfonts}
\usepackage{amssymb}
\usepackage{amscd}
\usepackage{amsthm}
\usepackage{layout}
\usepackage{color}

\theoremstyle{plain}
\theoremstyle{definition}
\newtheorem{theorem}{Theorem}[section]
\newtheorem{lemma}[theorem]{Lemma}
\newtheorem{proposition}[theorem]{Proposition}
\newtheorem{corollary}[theorem]{Corollary}
\newtheorem{definition}[theorem]{Definition}
\newtheorem{remark}[theorem]{Remark}

\newtheorem{notation}[theorem]{Notation}

\newtheorem{thm}{Theorem}

\newtheorem*{dfn*}{Definition}

\newtheorem*{cor*}{Corollary}

\newtheorem*{prp*}{Proposition}

\newtheorem*{rmk*}{Remark}

\newtheorem{qst}{Question}

\newcommand{\N}{\mathbb{N}}
\newcommand{\R}{\mathbb{R}}
\newcommand{\e}{\varepsilon}

\DeclareMathOperator{\supp}{supp}
\DeclareMathOperator{\ran}{ran}
\DeclareMathOperator{\scc}{succ}

\DeclareMathOperator{\dist}{dist}

\newcommand{\inn}{\in\mathbb{N}}
\newcommand{\de}{\delta}
\newcommand{\al}{\alpha}
\newcommand{\be}{\beta}
\newcommand{\es}{\emptyset}
\newcommand{\la}{\lambda}
\newcommand{\X}{X_{(\T,L,\e)}}
\newcommand{\Ft}{\tilde{\Phi}_{(\T, L, \e)}}
\newcommand{\F}{\Phi_{(\T, L, \e)}}
\newcommand{\C}{\mathcal{C}}

\DeclareMathOperator{\im}{Im}
\DeclareMathOperator{\T}{\mathcal{T}}
\DeclareMathOperator{\s}{\mathcal{S}}

\long\def\symbolfootnote[#1]#2{\begingroup%
\def\thefootnote{\fnsymbol{footnote}}\footnote[#1]{#2}\endgroup}

\begin{document}
\title{A hierarchy of Banach spaces with $C(K)$ Calkin Algebras}
\author[P. Motakis]{Pavlos Motakis}
\address{National Technical University of Athens, Faculty of Applied Sciences, Department
of Mathematics, Zografou Campus, 157 80, Athens, Greece}
\email{pmotakis@central.ntua.gr}
\author[D. Puglisi]{Daniele Puglisi}
\address{Department of Mathematics and Computer Sciences, University of Catania,  Catania, 95125, Italy (EU)}
\email{dpuglisi@dmi.unict.it}
\author[D. Zisimopoulou]{Despoina Zisimopoulou}
\address{National Technical University of Athens, Faculty of Applied Sciences, Department
of Mathematics, Zografou Campus, 157 80, Athens, Greece}
\email{dzisimopoulou@hotmail.com}

\date{}

\maketitle

\symbolfootnote[0]{\textit{2010 Mathematics Subject
Classification:} Primary 46B03, 46B25, 46B28}


\symbolfootnote[0]{This research was supported by program API$\Sigma$TEIA 1082.}

\symbolfootnote[0]{The second author was supported by ``National Group for Algebraic and Geometric Structures, and their Applications'' (GNSAGA - INDAM).}

\begin{abstract}
For every well founded tree $\mathcal{T}$ having a unique root such that every non-maximal node of it has countable infinitely many immediate successors, we construct a $\mathcal{L}_\infty$-space $X_{\mathcal{T}}$. We prove that for each such tree $\mathcal{T}$, the Calkin algebra of $X_{\T}$ is homomorphic to $C(\T)$, the algebra of continuous functions defined on $\mathcal{T}$, equipped with the usual topology. We use this fact to conclude that for every countable compact metric space $K$ there exists a $\mathcal{L}_\infty$-space whose Calkin algebra is isomorphic, as a Banach algebra, to $C(K)$.
\end{abstract}

\section*{Introduction}

The Calkin algebra of a Banach space $X$ is the quotient algebra $\mathcal{C}al(X) = \mathcal{L}(X)/\mathcal{K}(X)$ where $\mathcal{L}(X)$ denotes the algebra of all bounded linear operators defined on $X$ and $\mathcal{K}(X)$ denotes the ideal of all the compact ones. It is named after J. W. Calkin, who proved in \cite{C} that the only non-trivial closed ideal of the bounded linear operators on $\ell_2$ is the one of the compact operators. It is an important example of a unital algebra, for example an operator $U \in \mathcal{L}(X)$  is Fredholm if and only if the class $[U]$ of $U$ in the Calkin algebra is invertible. A question that arises is the following: let $A$ be  a Banach algebra, does there exist a Banach space $X$ so that the Calkin algebra of $X$ is isomorphic, as a Banach algebra, to $A$?

The very first Banach space for which the Calkin algebra was explicitly described is the Argyros-Haydon space $X_{AH}$ \cite{AH} whose main feature is that it has the ``scalar plus compact'' property, i.e. the ideal of the compact operators is of co-dimension one in the algebra of all bounded operators and hence its Calkin algebra is one-dimensional. The aforementioned construction is based on a combination of the methods used by W. T. Gowers and B. Maurey in \cite{GM1} and by J. Bourgain and F. Delbaen in \cite{BD}, hence the resulting space $X_{AH}$ is hereditarily indecomposable (HI) as well as a $\mathcal{L}_\infty$-space. It is worth mentioning that for every natural number $k$, by carefully taking $X_1,\dots, X_k$ versions of the Argyros-Haydon space, the Calkin algebra of the direct sum of these spaces is $k$-dimensional. In 2013 M. Tarbard \cite{T}, combining the techniques from \cite{AH} and \cite{GM}, provided an example of a Banach space $\mathfrak{X}_\infty$ such that its Calkin algebra is isometric, as a Banach  algebra, to the convolution algebra $\ell_1(\mathbb{N}_0)$.

Since the space $\ell_1$ occurs as a Calkin algebra, one may ask whether the same is true for $c_0$. To make the question more precise, does there exist a Banach space such that its Calkin algebra is isomorphic, as a Banach algebra, to $C(\omega)$? Or more generally, one may ask for what topological spaces $K$, the algebra $C(K)$ is isomorphic to the Calkin algebra of some Banach space.

The above question is the one under consideration in the present paper, in particular we prove the following:
\begin{thm}\label{a}
Let $\T$ be a well founded tree with a unique root, such that every non-maximal node of $\T$ has countable infinitely many immediate successors. Then there exists a $\mathcal{L}_\infty$-space $X_{\T}$ with the following properties.
\begin{itemize}

\item[(i)] The dual of $X_{\T}$ is isomorphic to $\ell_1$.

\item[(ii)] There exists a family of norm-one projections $(I_s)_{s\in\T}$ such that every operator defined on the space is approximated by a sequence of operators, each one of which is a linear combination of these projections plus a compact operator.

\item[(iii)] There exists a bounded, one-to-one and onto algebra homomorphism $\Phi: \C al(X_{\T}) \rightarrow C(\T)$, where $C(\T)$ denotes the algebra of all continuous functions defined on the compact topological space $\T$. In other words, the Calkin of $X_{\T}$ is isomorphic, as a Banach algebra, to $C(\T)$.




\end{itemize}
\end{thm}


As an application we obtain the following.
\begin{thm}\label{b}
For every countable compact metric space $K$ there exists a $\mathcal{L}_\infty$-space $X$, with $X^*$ isomorphic to $\ell_1$, so that its Calkin algebra is isomorphic, as a Banach algebra, to $C(K)$.
\end{thm}

The construction of the spaces is done recursively on the order of the tree $\T$. On the basic recursive step, i.e. in the case $\T$ is a singleton, the space $X_{\T}$ is in fact the Argyros-Haydon space. In the general case, the space $X_{\T}$ is the direct sum $(\sum\oplus X_{\T_n})_{AH}$, where the trees $\T_n$ have order smaller than the one of $\T$ and the outside norm is the Argyros-Haydon one, as it was defined by the third named author in \cite{Z}. Given the space $X_{AH}$ from \cite{AH} and the definition of the direct sum $(\sum\oplus X_n)_{AH}$ of a sequence of Banach spaces from \cite{Z}, the definition of the spaces $X_{\T}$ can be formulated quite easily, however the proofs involve many details from these constructions and are in some cases quite technical.

The paper is divided into five sections. The preliminary section is the most lengthy one and it mainly discusses the tools from \cite{AH} and \cite{Z} which are essential to obtain the result in this paper. The basic properties of constructions from these papers are presented while we also make some remarks not mentioned there. Moreover, some basic facts about trees are included. Section \ref{the definition} is devoted to the definition of the spaces $X_{\T}$ and the proof of some of their essential properties. Section \ref{the operators} focuses on the study of the operators defined on these spaces and concludes with the fact that for every such space $X_{\T}$ there exists a set of bounded norm-one projections $\{I_s:\;s\in\T\}$ so that every operator on $X_{\T}$ is approximated by  operators, each one of which is a linear combination of these projections plus a compact operator. In the fourth section, using the aforementioned result, we prove that the Calkin algebra of $X_{\T}$ is homomorphic to $C(\T)$, the algebra of continuous functions on $\T$. In the fifth, and final, section we use the properties of our construction to deduce the main result, i.e. that for every countable compact metric space $K$, the exists a Banach space whose Calkin algebra is isomorphic, as a Banach algebra to $C(K)$.

\section{Preliminaries}\label{preliminaries}
We begin with a preliminaries section which includes a lot of facts and estimations which concern Bourgain-Delbaen spaces, the Argyros-Haydon space as well as the Argyros-Haydon sum of a sequence of Banach spaces as it was defined in \cite{Z}. The section is unavoidably technical and extensive as our methods rely on some details concerning the just mentioned constructions. We also mention some basics concerning trees.

We start with the definition of a $\mathcal{L}_{\infty}$-space, which first appeared in \cite{LP}. We recall that for Banach spaces $Z,W$ and a constant $C>0$, we say that $Z$ is $C$-isomorphic to $W$, denoted as $Z\simeq^C W$, if there exists an onto linear isomorphism $T:Z\to W$ such that $\|T\|\|T^{-1}\|\leq C$.

\begin{definition}\label{BD}
We say that a separable Banach space $X$ is a $\mathcal{L}_{\infty,C}$-space where $C>0$ is a constant, if there exists a strictly increasing sequence $(Y_n)_{n\in\N}$ of subspaces of $X$ such that $Y_n\simeq^C\ell_{\infty}(\dim Y_n)$ for every $n\in\N$ and $X=\overline{\cup_{n\in\N}Y_n}$.
\end{definition}

Using the Bourgain-Delbaen method of constructing $\mathcal{L}_{\infty}$-spaces $X$ \cite{BD} we can specify the constant $C>0$ as well as the sequence $(Y_n)_{n\in\N}$ that correspond to $X$ in Definition \ref{BD}. We say that a space $X$ is a BD-$\mathcal{L}_{\infty}$-space if it is constructed via the BD-method. Two fixed parameters $0<a<1$, $0<b<\frac{1}{2}$ with $a+b>1$ are used and there exist
\begin{enumerate}
\item[(i)] a sequence $(\Delta_n)_{n\in\N}$ of pairwise disjoint subsets of $\N$, where we denote their union by $\Gamma$ and we set $\Gamma_n=\cup_{i=1}^n\Delta_i$ for each $n\in\N$,
\item[(ii)] linear extension operators $i_n:\ell_{\infty}(\Gamma_n)\to \ell_{\infty}(\Gamma)$ with $\|i_n\|\|i_n^{-1}\|\leq \frac{1}{1-2b}$ for every $n\in\N$,
\end{enumerate}
such that $X=\overline{\cup_{n\in\N}Y_n}$, where $Y_n=i_n(\ell_{\infty}(\Gamma_n))$ for every $n\in\N$. In particular, the BD-space $X$ is a $\mathcal{L}_{\infty,C}$-space where $C=\frac{1}{1-2b}$.

\subsection{The space $X_{AH}$}

We denote by $\mathfrak{X}_{AH}$ the Argyros Haydon space constructed in \cite{AH}. The space $X_{AH}$ is a separable HI $\mathcal{L}_{\infty,2}$-space such that $\mathfrak{X}_{AH}^*\simeq^2\ell_1$. The construction is based on two fixed strictly increasing sequences of natural numbers $(m_j,n_j)_{j\in\N}$ (with $m_1\geq 4$) and it is a generalization of the BD method for parameters $a=1$ and using instead of $b$ the sequence $(\frac{1}{m_j})_{j\in\N}$. In particular, there exists a sequence $(\Delta_n)_{n\in\N}$ and linear operators $i_n$ as above having the following properties:
\begin{enumerate}
\item[(i)] To each
element $\gamma\in\Gamma$ are assigned the rank of
$\gamma$, $rank(\gamma)$, the weight of $\gamma$, $w(\gamma)$, and
the age of $\gamma$, $a(\gamma)$ so that
\begin{enumerate}
\item $rank(\gamma)=n$ whenever $\gamma\in\Delta_n$,
\item $w(\gamma)=m_j$ for some $j\leq n$ whenever $rank(\gamma)=n$,
\item $a(\gamma)=a\leq n_j$,
 \end{enumerate}
\item[(ii)] $\|i_n\|\|i_n^{-1}\|\leq 2$.
\end{enumerate}

Moreover, the space $\mathfrak{X}_{AH}$ admits a FDD $(M_n)_{n\in\N}$ where $M_n=i_n(\ell_{\infty}(\Delta_n))$ is isometric to $\ell_{\infty}(\Delta_n)$.
Using the FDD we define the range of an element $x$ in $\mathfrak{X}_K$ as the minimum interval $I$ of $\N$ such that $x\in\sum_{n\in I}\oplus M_n$. A bounded sequence $(x_k)_k$ in $\mathfrak{X}_K$ is called a block sequence if $\max\ran x_k<\min\ran x_{k+1}$ for every $k$.

\subsection{AH-$\mathcal{L}_{\infty}$ sums of separable Banach spaces}
In this subsection we remind the basic ingredients from \cite{Z} of constructing sums $(\sum_n\oplus X_n)_{AH}$ where $(X_n)_n$ is a sequence of separable Banach spaces. We start with the following notation.

\begin{notation}\label{terminology}
Let $(E_n,\|\cdot\|_{E_n})_{n=1}^{\infty}$ be sequences of
separable Banach spaces. By $(\sum_{n}\oplus E_n)_\infty$ we denote the space of vector elements $\overrightarrow{x}=(x_n)_{n=1}^{\infty}$ such that the $n$-th
coordinate of $\overrightarrow{x}$, is in $E_n$ for all $n\in\N$ and $\|\overrightarrow{x}\|_\infty = \sup\|x_n\|_{E_n}$ is finite. By $c_{00}(\sum_{n}\oplus E_n)$ we denote
the subspace of $(\sum_{n}\oplus E_n)_\infty$ consisting of all
$\overrightarrow{x}=(x_n)_{n=1}^{\infty}$ for which there exists $n_0\in\N$ with the property that $x_n=0$
for every $n\geq n_0$. For a vector
$\overrightarrow{x}=(x_n)_{n=1}^{\infty}\in c_{00}(\sum_{n}\oplus
E_n)$ we define the support of $x$ as \[\supp x=\{n \inn:\ x_n\neq
0\}\]
For every finite interval $J\subset\N$ we denote by $R_{J}$
the natural restriction map \[R_{J}:\left(\sum_{n}\oplus
E_n\right)_\infty\to\left(\sum_{n\in J}\oplus E_n\right)_{\infty}\] defined as
$R_{J}(\overrightarrow{x})=(x_n)_{n\in J}$ for
every $\overrightarrow{x}=(x_n)_{n=1}^{\infty}\in
c_{00}(\sum_{n}\oplus E_n)$. For $I,J$ subsets of $\N$ we say that
$I,J$ are successive denoted as $I<J$ if $\max I < \min J$. For
$\overrightarrow{x},\overrightarrow{y},\overrightarrow{z}$ vector
elements of $c_{00}(\sum_{n}\oplus E_n)$ such that $\supp
\overrightarrow{x}<\supp \overrightarrow{y}<\supp
\overrightarrow{z}$ we denote by
$(\overrightarrow{x},\overrightarrow{y},\overrightarrow{z})$ the
vector
$\overrightarrow{x}+\overrightarrow{y}+\overrightarrow{z}\in
c_{00}(\sum_{n}\oplus E_n)$.

By $(\sum_{n=1}^{\infty}\oplus E_n)_1$ we denote the space of vector elements $\overrightarrow{x}=(x_n)_{n=1}^{\infty}$ such that the $n$-th
coordinate of $\overrightarrow{x}$, is in $E_n$ and $\|\overrightarrow{x}\|_1 = \sum_n\|x_n\|_{E_n}$ is finite. If
the elements of $c_{00}(\sum_{n}\oplus E_n)$ are considered to be
functionals (i.e $E_n$ are dual spaces) we use letters as
$\overrightarrow{f},\overrightarrow{g},\overrightarrow{h}$, etc
for their representation. For $\overrightarrow{x}\in
(\sum_{n}\oplus E_n)_\infty$ and $\overrightarrow{f}\in
c_{00}(\sum_{n}\oplus E_n^*)$ we denote by
$\overrightarrow{f}(\overrightarrow{x})$ the inner product $\sum_n
f_n(x_n)$.
\end{notation}

We continue with the definition of spaces
$(\sum_{n=1}^{\infty}\oplus X_n)_{BD}$ for a sequence
$(X_n,\|\cdot\|_n)_{n\in\N}$ of separable Banach spaces.

\begin{definition}\label{mathcal}
Let $(X_n,\|\cdot\|_n)_{n\in\N}$ be a sequence of separable Banach
spaces. A Banach space $\mathcal{Z}$ is called a Bourgain
Delbaen(BD)-$\mathcal{L}_{\infty,C}$-sum of the sequence
$(X_n,\|\cdot\|_n)_n$, denoted as
$\mathcal{Z}=(\sum_{n=1}^{\infty}\oplus X_n)_{BD}$, if there
exists a sequence $(\Delta_n)_{n\in\N}$ of finite, pairwise
disjoint subsets of $\N$ and the following hold:
\begin{enumerate}
\item[(i)] The space $\mathcal{Z}$ is a subspace of
$(\sum_{n=1}^{\infty}\oplus(X_n\oplus\ell_{\infty}(\Delta_n))_{\infty})_{\infty}=\mathcal{Z}_{\infty}$.
\item[(ii)]  For every $n$ there exists a linear extension operator
\[i_n:\left(\sum_{k=1}^n\oplus\left(X_k\oplus\ell^{\infty}(\Delta_k)\right)_\infty\right)_\infty\rightarrow\left(\sum_{n=1}^{\infty}\oplus\left(X_n\oplus\ell_{\infty}(\Delta_n)\right)_{\infty}\right)_{\infty}\]
such that

\begin{enumerate}
\item $\|i_n\|\leq C$ for every $n\in\mathbb{N}$. \item Each
$\overrightarrow{x}\in\left(\sum_{k=1}^n\oplus\left(X_k\oplus\ell_{\infty}(\Delta_k)\right)_\infty\right)_\infty$
satisfies the following:
\begin{enumerate}
\item[($\iota$)] $R_{[1,n]}(i_n(\overrightarrow{x}))=\overrightarrow{x}$
while $R_{(n,\infty)}(i_n(\overrightarrow{x}))$ is an element of
$\left(\sum_{k=n+1}^{\infty}\oplus(\{0\}\oplus\ell_{\infty}(\Delta_k))_{\infty}\right)_\infty$.
\item[($\iota\iota$)]
$i_l(R_{[1,l]}i_n(\overrightarrow{x}))=i_n(\overrightarrow{x})$
for every $l\geq n+1$
\end{enumerate}
\end{enumerate}
\item[(iii)] Setting
$Y_n=i_n\left[\left(\sum_{k=1}^n\oplus(X_k\oplus\ell^{\infty}(\Delta_k))_\infty\right)_\infty\right]$,
the union $\cup_n Y_n$ is dense in $\mathcal{Z}$.
\end{enumerate}
\end{definition}
A space $\mathcal{Z}=(\sum_{n=1}^{\infty}\oplus X_n)_{BD}$ can be
obtained by modifying the original Bourgain-Delbaen
$\mathcal{L}_{\infty}$ method of construction in \cite{BD} as it
was described in \cite{Z}. In particular, applying this modification on the BD-method of Argyros-Haydon we can construct
BD-$\mathcal{L}_{\infty,2}$ sums $(\sum_{n=1}^{\infty}\oplus
X_n)_{BD}$ of separable Banach spaces $X_n$, denoted as
$(\sum\oplus X_n)_{AH}$. The next result is proved in (\cite{Z}).

\begin{proposition}\label{dualZAH}
Let $(X_n,\|\cdot\|_n)_{n\in\N}$ be a sequence of separable Banach
spaces. The space $\mathcal{Z}=(\sum\oplus X_n)_{AH}$ has the following properties:
\begin{enumerate}

\item[(i)] $\mathcal{Z}=(\sum\oplus X_n)_{AH}$ admits a
shrinking Schauder Decomposition $(Z_n)_{n\in\N}$ such that each
$Z_n=i_n[(X_n\oplus\ell_{\infty}(\Delta_n))_{\infty}]$.

\item[(ii)]
Every horizontally block subspace of $\mathcal{Z}$ is HI.

\item[(iii)]
The dual $\mathcal{Z}^*$ is 2-isomorphic to
$(\sum_{n=1}^{\infty}\oplus(X_n^*\oplus\ell_1(\Delta_n))_1)_1$.

 \end{enumerate}

\end{proposition}

\begin{definition}\label{range}
For an element $z$ of $\mathcal{Z}=(\sum\oplus X_n)_{AH}$ we define the range of $z$, denoted by $\ran z$, as the minimum interval $I\subset\N$ such that $z\in\sum_{n\in I}\oplus Z_n$.
\end{definition}


 Since $\mathcal{Z}$ is a subspace of $\mathcal{Z}_{\infty}$, we can consider the restriction mappings $R_{[1,n]}:\mathcal{Z}\to \left(\sum_{k=1}^n\oplus\left(X_k\oplus\ell_{\infty}(\Delta_k)\right)_\infty\right)_{\infty}$.
We denote by $P_{[1,n]}:\mathcal{Z}\to\mathcal{Z}$ the projections associated with the Schauder Decomposition $(Z_n)_{n\in\N}$ that are defined as $P_{[1,n]}=i_n\circ R_{[1,n]}$. We write $P_n$ instead of $P_{\{n\}}$ and for every $k\leq m$ we define $P_{(k,m]}=\sum_{i=k+1}^m P_i=i_m\circ R_{[1,m]}-i_k\circ R_{[1,k]}$.

We also identify $\left(\sum_{n=k}^m\oplus (X_n\oplus\{0\})_\infty\right)_{\infty}$ with $(\sum_{n=k}^m\oplus X_n)_{\infty}$ and similarly $(\sum_{n=k}^m\oplus \left(\{0\}\oplus\ell_{\infty}(\Delta_n)\right)_\infty)_{\infty}$ with $(\sum_{n=k}^m\oplus \ell_{\infty}(\Delta_n))_{\infty}$. For $\gamma\in\Delta_n$ we denote by $e_{\gamma}^*$ the usual vector
element of $\ell_1(\Delta_n)$. We extend each element
$b^*=\sum_{\gamma}a_{\gamma}e_{\gamma}^*\in(\sum_{n=k}^m\oplus\ell_1(\Delta_n))_1$
to a functional $\overrightarrow{b}^*:\mathcal{Z}\to\R$ defined as
$\overrightarrow{b}^*=b^*\circ R_{[1,n]}$. Similarly we extend
every $\overrightarrow{f}\in(\sum_{n=k}^m\oplus X_n^*)_1$ to a
functional $\overrightarrow{f}:\mathcal{Z}\to\R$ 

The construction of $\mathcal{Z}=(\sum_{n=1}^{\infty}\oplus
X_n)_{AH}$ is based on the same parameters $(m_j,n_j)_{j=1}^{\infty}$ of $X_{AH}$. More precisely, for every $n\in\N$ each set $\Delta_{n+1}$ is the union of pairwise
disjoint sets $\Delta_{n+1}=\Delta_{n+1}^0\cup\Delta_{n+1}^1$ satisfying the following:
\begin{enumerate}
\item[(i)]
For every $\gamma\in\Delta_{n+1}^0$ there is $\overrightarrow{f}\in\left(\sum_{k=1}^{n}\oplus X_k^*\right)_1$ with $\|\overrightarrow{f}\| \leq 1$ so that
\begin{equation}\label{tadaaah}
\overrightarrow{e}_\gamma^* = \overrightarrow{e}_\gamma^*\circ P_{n+1} + \frac{1}{m_1}\overrightarrow{f}.
\end{equation}
\item[(ii)] The set $\Delta_{n+1}^1$ is defined similarly as the set $\Delta_n$ of $X_{AH}$ and consists of elements with rank, weight, age depending on $(m_j,n_j)_{j=1}^{\infty}$. Moreover, as it was stated in \cite{Z}, for every
$\gamma\in\Delta_{n+1}^1$ with weight $w(\gamma)=m_j$ there exists a
family $\{p_i,q_i,\xi_i,\overrightarrow{b}_i^*\}_{i=1}^a$, called
the analysis of $\gamma$, such that:

\begin{enumerate}

\item $0\leq p_1 < q_1 < p_2 < q_2 <\cdots<p_a<q_a = n$,


\item
$\xi_i\in\Delta_{q_i+1}^1$ with $w(\xi_i)=m_j$ and
$b_i^*\in(\sum_{k=p_i}^{q_i}\oplus\ell_1(\Delta_k))_1$,

\item
$\overrightarrow{e_{\gamma}}^*=\sum_{i=1}^a\overrightarrow{e_{\xi_i}}^*\circ
P_{\{q_i+1\}}+\frac{1}{m_j}\sum_{i=1}^a\overrightarrow{b_i}^*\circ
P_{(p_i,q_i]}$.
\end{enumerate}

\end{enumerate}

\begin{remark}\label{not exaaactly the same but still pretty close}
In \cite{Z} \eqref{tadaaah} is actually slightly different. More precisely, in that case for every $n\in\N$ and $\gamma\in\Delta_{n+1}^0$, we have that $\overrightarrow{e}_\gamma^* = \overrightarrow{e}_\gamma^*\circ P_{n+1} + \overrightarrow{f}$, i.e. the constant $1/m_1$ is missing. All results from \cite{Z} also hold with this modification, the only difference being that in various cases some constants have to be adjusted. This modification is used only to prove Proposition \ref{all the money}, which is essential to obtaining the main result of this paper.
\end{remark}

For the rest of this subsection we fix a sequence of separable Banach spaces $(X_n,\|\cdot\|_n)_{n\in\N}$ and let $\mathcal{Z}=(\sum_{n=1}^{\infty}\oplus
X_n)_{AH}$ with Schauder Decomposition $(Z_n)_{n\in\N}$ be as stated in Proposition \ref{dualZAH}.

\begin{lemma}\label{analyze back}
Let $n,q\inn$ with $q>n$ and $\gamma\in\Delta_q^1$ with $w(\gamma) = m_j$ for some $j\inn$. Consider the functional $g:\mathcal{Z}\rightarrow \mathbb{R}$ with $g = \overrightarrow{e}_{\gamma}^*\circ P_{[1,n]}$. Then one of the following holds:
\begin{itemize}

\item[(i)] $g = 0$,

\item[(ii)] there are $p_1\leq n$ and $\overrightarrow{b}^*\in\left(\sum_{k=1}^n\oplus\ell_1(\Delta_k)\right)_1$ with $\|\overrightarrow{b}^*\|\leq 1$ so that $g = \frac{1}{m_j}\overrightarrow{b}^*\circ P_{(p_1,n]}$ or

\item[(iii)] there are $p_0<p_1\leq n$ and $\gamma'\in\Delta_{p_0}$ and $\overrightarrow{b}^*$ as before so that $g = \overrightarrow{e}_{\gamma'}^* + \frac{1}{m_j}\overrightarrow{b}^*\circ P_{[p,n]}$.

\end{itemize}
\end{lemma}

\begin{proof}
Let $\overrightarrow{e_{\gamma}}^*=\sum_{i=1}^a\overrightarrow{e_{\xi_i}}^*\circ
P_{\{q_i+1\}}+\frac{1}{m_j}\sum_{i=1}^a\overrightarrow{b_i}^*\circ
P_{(p_i,q_i]}$ as in (c) above. Note that as in \cite[Proposition 4.5]{AH} we have that for every $1\leq i_0 < a$ the following holds:
\begin{equation}\label{mama mia}
\overrightarrow{e_{\gamma}}^* = \overrightarrow{e_{\xi_{i_0}}}^* + \sum_{i=i_0 + 1}^a\overrightarrow{e_{\xi_i}}^*\circ
P_{\{q_i+1\}}+\frac{1}{m_j}\sum_{i=i_0 + 1}^a\overrightarrow{b_i}^*\circ
P_{(p_i,q_i]}
\end{equation}

Let us first assume that $n < q_1+1$. If $n\leq p_1$ then we easily conclude that $\overrightarrow{e}_{\gamma}^*\circ P_{[1,n]} = 0$ and the first assertion holds. Otherwise $p_1<n$ and $\overrightarrow{e}_{\gamma}^*\circ P_{[1,n]} = \frac{1}{m_j}\overrightarrow{b}_i^*\circ P_{(p_1,n]}$, i.e. the second assertion holds.

Let us now assume that $q_1 + 1 \leq n$ and set $i_0 = \max \{i: q_i + 1\leq n\}$. Since $q>a$, by property (a) it follows that $i_0 < a$ and so, using \eqref{mama mia}, we obtain
\begin{equation*}
\overrightarrow{e}_{\gamma}^*\circ P_{[1,n]} = \overrightarrow{e_{\xi_{i_0}}}^*\circ P_{[0,n]} + \frac{1}{m_j}\overrightarrow{b}_{i_0+1}^*\circ P_{(p_{i_0+1},n]}
\end{equation*}
(where if $p_{i_0+1} > n$ then the last part of the right hand side in the above inequality is zero). Since $\xi_{i_0}\in \Delta_{q_{i_0}+1}^{1}$ and $i_0+1\leq n$ we conclude that $\overrightarrow{e_{\xi_{i_0}}}^*\circ P_{[0,n]} = \overrightarrow{e_{\xi_{i_0}}}^*$ and hence $g = \overrightarrow{e_{\xi_{i_0}}}^* + \frac{1}{m_j}\overrightarrow{b}_{i_0+1}^*\circ P_{(p_{i_0+1},n]}$.
\end{proof}

\begin{lemma}\label{isometryZn} The space $Z_n$ is isometric to $(X_n\oplus\ell_{\infty}(\Delta_n))_{\infty}$ for every $n\in\N$. More precisely, the operator $i_n$ restricted onto $\left(X_n\oplus\ell_\infty(\Delta_n)\right)_\infty$ (viewed as a subspace of $\left(\sum_{k=1}^n\oplus\left(X_k\oplus\ell_\infty(\Delta_k)\right)_\infty\right)_\infty$), is an isometry.
\end{lemma}
\begin{proof}
Let $z=i_n(\overrightarrow{x})\in Z_n$ where
$\overrightarrow{x}\in(X_n\oplus\ell_{\infty}(\Delta_n))_{\infty}$. We will prove that $\|z\|=\|x\|_{\infty}$.
Observe that by Definition \ref{mathcal}
$R_{(n,\infty)}(z)\in(\sum_{k>n}\oplus\ell_{\infty}(\Delta_n))_{\infty}$.
Let $k>n$ and $\gamma\in\Delta_k^1$ with $w(\gamma)=m_j$ and
analysis $\{p_i,q_i,\xi_i,\overrightarrow{b}_i^*\}_{i=1}^a$. Note that there exists at most one $1\leq i\leq a$ such that either $n=q_i+1$ or $n\in (p_i,q_i]$. In the first case we have that
$\overrightarrow{e_{\gamma}}^*(z)=\overrightarrow{e_{\xi_i}}^*(P_{\{q_i+1\}}z)=e_{\xi_i}^*(x)\leq\|x\|_{\infty}$
and similarly in the second case
$\overrightarrow{e_{\gamma}}^*(z)=\frac{1}{m_j}\overrightarrow{b}^*_i(P_{(p_i,q_i]}z)=\frac{1}{m_j}b^*_i(x)\leq
\|x\|_{\infty}$. Hence $\|R_{(n,\infty)}z\|\leq \|x\|_{\infty}$ and since $\|x\|_{\infty}=\|R_{[1,n]}z\|$ it follows that $\|z\|=\|x\|_{\infty}$ as promised.
\end{proof}

If by $j_n: X_n\rightarrow \left(\sum_{k=1}^n\oplus\left(X_n\oplus\ell_\infty(\Delta_k)\right)_\infty\right)_\infty$ we denote the natural embedding, we immediately obtain the following.
\begin{corollary} \label{isometryXn}
The map $i_n\circ j_n: X_n\rightarrow \mathcal{Z}$ is an isometric embedding and hence the space $i_n\circ j_n[X_n]$ is isometric to $X_n$ for every $n\in\N$.
\end{corollary}

Let $\pi_n:\left(\sum_{k=1}^n\oplus(X_n\oplus\ell_{\infty}(\Delta_n))_{\infty}\right)_\infty\rightarrow X_n$ denote the natural restriction onto the coordinate $X_n$.


\begin{lemma} For every $n\inn$, the map $I_n:\mathcal{Z}\rightarrow\mathcal{Z}$ with $I_n = i_n\circ j_n \circ \pi_n\circ R_{[1,n]}$ is a norm-one projection onto $i_n\circ j_n [X_n]$ with $\ker I_n = i_n[\ell_{\infty}(\Delta_n)]\oplus\sum_{k\neq n}\oplus
Z_k$ (where $\ell_\infty(\Delta_n)$ is viewed as a subspace of $\left(\sum_{k=1}^n\oplus\left(X_k\oplus\ell_\infty(\Delta_k)\right)_\infty\right)_\infty$ in the canonical way).\label{constant projectionIn}
\end{lemma}

\begin{proof}
The fact that $I_n$ is a projection, with the image and kernel mentioned above, follows from its definition. To see that $\|I_n\| = 1$ let $z\in\mathcal{Z}$ and $\overrightarrow{x}=R_{[1,n]}(z)$. Note that $I_n(z)=i_n\circ j_n\circ \pi_n(\overrightarrow{x})$ and by Corollary \ref{isometryXn} we have that $\|I_n(z)\|\leq \|j_n\circ\pi_n(\overrightarrow{x})\| \leq \|\overrightarrow{x}\| \leq \|z\|$ and hence $\|I_n\|\ = 1$.
\end{proof}

\begin{remark}\label{difference compact}
Note that for every $n\inn$ the difference $P_{n} - I_n$ is a finite rank operator. More precisely, $(P_{n} - I_n)(\mathcal{Z}) = i_n[\ell_{\infty}(\Delta_n)]$ (where $\ell_\infty(\Delta_n)$ is viewed as a subspace of $\left(\sum_{k=1}^n\oplus\left(X_k\oplus\ell_\infty(\Delta_k)\right)_\infty\right)_\infty$ in the canonical way).
\end{remark}

\begin{remark}\label{sums of ops}
For $n\in\N$ and a bounded linear $T_n:X_n\rightarrow X_n$ we may identify $T_n$ with the map $\tilde{T}_n = i_n\circ j_n\circ T_n\circ \pi_n\circ R_{[1,n]}$ which is defined on $\mathcal{Z}$. It follows that $\|T_n\| = \|\tilde{T}_n\|$. Also, for $n\in\N$ and bounded linear operators $T_k:X_k\rightarrow X_k$, $k=1,\ldots,n$ we may define the operator $\sum_{k=1}^n\oplus T_k = \sum_{k=1}^n\tilde{T}_k$. Proposition \ref{all the money} below provides a relation of the norm of $\sum_{k=1}^n\oplus T_k$, in the Calkin algebra of $\mathcal{Z}$, to the norms of the $T_k$, $k=1,\ldots,n$.
\end{remark}

\subsection{AH(L)-$\mathcal{L}_{\infty}$ sums of separable Banach spaces for $L$ an infinite subset of the natural numbers }

The construction of $\mathfrak{X}_{AH}$ is based on a sequence of parameters $(m_j,n_j)_{j\in\N}$ satisfying certain lacunarity conditions which are preserved under taking a subsequence $(m_j,n_j)_{j\in L}$, where $L$ is an infinite subset of $\N$. Hence for every such $L$ the space $\mathfrak{X}_{AH(L)}$ can been defined using as parameters the sequence $(m_j,n_j)_{j\in L}$.

In a similar manner, such a sequence $(m_j,n_j)_j$ is used for constructing the Argyros-Haydon  sum of a sequence of separable Banach spaces $(X_n)_n$. Using an infinite subset of the natural numbers $L$ and as parameters the sequence $(m_j,n_j)_{j\in L}$ we define the space $(\sum_{n=1}^{\infty}\oplus X_n)_{AH(L)}$. Note also that in \eqref{tadaaah} the constant $1/m_1$ is now replaced with $1/m_{\min L}$.

The statements from the following remark follow from the corresponding proofs in \cite[Proposition 3.2, Theorems 3.4 and 3.5, Proposition 5.11]{AH}.
\begin{remark}\label{constantAH}
Let $L$ be an infinite subset of the natural numbers. If $\e = 2/(m_{\min L}-2)$, then the following more precise estimations are satisfied for the space $\mathfrak{X}_{AH(L)}$:
\begin{enumerate}

\item[(i)] The extension operators $i_n$ have norm at most $1+\e$.

\item[(ii)] The space $\mathfrak{X}_{AH(L)}$ is a $\mathcal{L}_{\infty,1+\e}$-space.

\item[(iii)] The dual of $\mathfrak{X}_{AH(L)}$ is $(1+\e)$-isomorphic to $\ell_1$.

\end{enumerate}

\end{remark}
We recall the following result from \cite{AH} that is needed for the sequel.

\begin{proposition}\label{upper}
Let $L$ be an infinite subset of $\N$ and $(z_k)_k$ be a bounded block sequence in $X_{AH(L)}$. Then there exists an infinite subset $\tilde{L}$ of $L$ such that for every $j\in\tilde{L}$ there exists a subsequence $(z_{k_i})_i$ of $(z_k)_k$ satisfying \[\left\|\sum_{i=1}^{n_j}z_{k_i}\right\|\geq\frac{1}{2m_j}\sum_{i=1}^{n_j}\|z_{k_i}\|.\]
\end{proposition}

Just as above, the statements from the remark below also follow from the corresponding proofs  in \cite[Lemma 2.6, Propositions 3.1 and 5.1, Corollary 5.15]{Z}.
\begin{remark}\label{constant}
Let $L$ be an infinite subset of $\N$ and $(X_n)_n$ be a sequence of separable Banach spaces. If $\e = 2/(m_{\min L}-2)$, then the following more precise estimations are satisfied for the space $(\sum_{n=1}^{\infty}\oplus X_n)_{AH(L)}$:
\begin{enumerate}

\item[(i)] The extension operators $i_n$ have have norm at most $1+\e$.

\item[(ii)] The space $(\sum_{n=1}^{\infty}\oplus X_n)_{AH(L)}$ is a BD-$\mathcal{L}_{\infty,1+\e}$ sum of $(X_n)_n$.

\item[(iii)] The dual of $(\sum_{n=1}^{\infty}\oplus X_n)_{AH(L)}$ is $(1+\e)$-isomorphic to the space $$\left(\sum_{n=1}^{\infty}\oplus \left(X_n^*\oplus\ell_1\left(\Delta_n\right)\right)_1\right)_1$$ where $\Delta_n$ are finite sets.

\end{enumerate}
\end{remark}

The next result is proved in \cite[Proposition 6.3]{Z}, where the constant $1/m_{\min L}$ results from the adjustment stated in Remark \ref{not exaaactly the same but still pretty close}.
\begin{proposition}\label{upper AH sum}
Let $(X_n, \|\cdot\|_n)_{n\in\N}$ be a sequence of separable Banach
spaces, $L$ be an infinite subset of $\N$, $\mathcal{Z}=(\sum_{n=1}^{\infty}\oplus
X_n)_{AH(L)}$ and $(z_k)_k$ be a bounded block sequence in $\mathcal{Z}$. Then there exists an infinite subset $\tilde{L}$ of $L$ such that for every $j\in \tilde{L}$ there exists a subsequence $(z_{k_i})_{i\in\N}$ of $(z_k)_k$ satisfying
\[\left\|\sum_{i=1}^{n_j}z_{k_i}\right\|\geq\frac{1}{m_{\min L}}\cdot\frac{1}{2m_j}\sum_{k=1}^{n_j}\|z_{k_i}\|.\]
\end{proposition}

The following Proposition essentially states that the norm of diagonal operators defined on $\mathcal{Z} = \left(\sum_n\oplus X_n\right)_{AH(L)}$, in the Calkin algebra of $\mathcal{Z}$, is a sufficient approximation of the supremum of the norms of the operator when restricted onto the coordinates $X_n$. This is the only part of this paper where the modification of \eqref{tadaaah} stated in Remark \ref{not exaaactly the same but still pretty close} is needed.

\begin{proposition}\label{all the money}
Let $(X_n\|\cdot\|_n)_{n\in\N}$ be a sequence of separable Banach spaces, $L$ be an infinite subset of $\N$, $\mathcal{Z}=(\sum_{n=1}^{\infty}\oplus X_n)_{AH(L)}$. Let also $n\inn$, $T_k:X_k\rightarrow X_k$ be bounded linear operators for $k=1,\ldots,n$ and $\lambda\in\mathbb{R}$. Define $T:\mathcal{Z}\rightarrow\mathcal{Z}$ with $T = \sum_{k=1}^n\oplus T_k + \lambda P_{(n,+\infty)}$ (see Remark \ref{sums of ops}). Then there exists a compact operator $K:\mathcal{Z}\rightarrow\mathcal{Z}$ so that
\begin{equation*}
\|T - K\| \leq \left(1 + \frac{4}{m_{\min L} - 2}\right)\max\left\{\max_{1\leq k \leq n}\|T_k\|, |\lambda|\right\}.
\end{equation*}
\end{proposition}

\begin{proof}
Consider the operators $$S_n^1, S_n^2:\left(\sum_{k=1}^n\oplus\left(X_k\oplus\ell_\infty(\Delta_k)\right)_\infty\right)_\infty\rightarrow \left(\sum_{k=1}^n\oplus\left(X_k\oplus\ell_\infty(\Delta_k)\right)_\infty\right)_\infty$$ such that if $x = (x_k,z_k)_{k=1}^n$ then $S_n^1x = (x_k,0)_{k=1}^n$ and $S_n^2x = (0,z_k)_{k=1}^n$. Define
$$A_n^1 = i_n\circ S_n^1\circ R_{[1,n]} \quad \mbox{and} \quad A_n^2 = i_n\circ S_n^2\circ R_{[1,n]}.$$
Observe that $A_n^2$ is a finite rank operator. Also note that $P_n = A_n^1 + A_n^2$ and hence $P_n - A_n^1$ is a finite rank operator as well.

Define $K = (P_n - A_n^1)\circ\sum_{k=1}^n\oplus T_k - \lambda A_n^2$ which is a finite rank operator. Then
\begin{equation}\label{i label this}
T-K = A_n^1\circ\sum_{k=1}^n\oplus T_k + \lambda\left(P_{(n,+\infty)}+A_n^2\right).
\end{equation}
We shall show that $\|T - K\| \leq (1+\delta)\max\left\{\max\|T_k\|, |\lambda|\right\}$, where $\delta = 4/(m_{\min L} - 2)$. Let $x$ be an element of $\mathcal{Z}$ with $\|x\| = 1$ and consider $x = (x_k,z_k)_{k=1}^\infty$ as a vector in $\mathcal{Z}_\infty = \left(\sum_{k=1}^\infty\left(X_k\oplus\ell_\infty(\Delta_k)\right)_\infty\right)_\infty.$ Observe that
\begin{equation}\label{i label that}
\left(A_n^1\circ\sum_{k=1}^n\oplus T_k\right)x = i_n\left((T_kx_k,0)_{k=1}^n\right) \quad\mbox{and}\quad \lambda A_n^2x = i_n\left((0,\lambda z_k)_{k=1}^n\right)
\end{equation}
Set $(T-K)x = y$, we shall prove that $\|y\| \leq (1 + \delta)\max\left\{\max\|T_k\|, |\lambda|\right\}.$  By \eqref{i label this}, \eqref{i label that} and $P_n = i_n\circ R_n$ we conclude that
\begin{equation}\label{i label those}
y = i_n\left((T_kx_k,\lambda z_k)_{k=1}^n\right) + \lambda(x - i_n\circ R_nx)
\end{equation}
Write $y = (y_k,w_k)_{k=1}^\infty$ as a vector in $\mathcal{Z}_\infty$. Note that $R_n(x - i_n\circ R_nx) = 0$, therefore by \eqref{i label those} we obtain that $R_ny = R_n\circ i_n\left((T_kx_k,\lambda z_k)_{k=1}^n\right) = (T_kx_k,\lambda z_k)_{k=1}^n$ and hence
\begin{equation*}\label{i label thee}
\left\|R_ny\right\| \leq \max\left\{\max\|T_k\|,|\lambda|\right\}.
\end{equation*}
Also, by \eqref{i label those} and Definition \ref{mathcal}(ii)(b)($\iota$) we obtain that for $k>n$ we have that $y_k = \lambda x_k$. All that remains to be shown is that for $q>n$, $\|w_q\| \leq (1 + \delta)\max\left\{\max\|T_k\|, |\lambda|\right\}.$ Fix $q>n$, it suffices to show that $|\overrightarrow{e}_\gamma(y)| \leq (1 + \delta)\max\left\{\max\|T_k\|, |\lambda|\right\}$ for all $\gamma\in\Delta_q$. To that end, let $\gamma\in \Delta_q$.


Using \eqref{i label those} rewrite $y$ as $$y = i_n\left((T_kx_k - \lambda x_k,0)_{k=1}^n\right) + \la x.$$ We conclude that
\begin{eqnarray}
|\overrightarrow{e}_\gamma(y)| &=& \left| \lambda \overrightarrow{e}_\gamma(x) + \overrightarrow{e}_\gamma\left(i_n\left((T_kx_k - \lambda x_k,0)_{k=1}^n\right)\right) \right|\nonumber\\
&\leq& |\lambda| + \left| \overrightarrow{e}_\gamma\left(i_n\left((T_kx_k - \lambda x_k,0)_{k=1}^n\right)\right) \right| \label{i label this upon thee}
\end{eqnarray}
We distinguish two cases concerning $\gamma$. We first study the case in which $\gamma\in\Delta_q^0$. In this case, there is $\overrightarrow{f}\in\left(\sum_{k=1}^{q-1}\oplus X_k^*\right)_1$ with $\|\overrightarrow{f}\|\leq 1$, so that $\overrightarrow{e}_{\gamma}^* = \overrightarrow{e}_{\gamma}^*\circ P_{\{q\}} +  \frac{1}{m_{\min L}}\overrightarrow{f}$ (recall that in the construction of $(\sum_k\oplus X_k)_{AH(L)}$, in \eqref{tadaaah} the constant $1/m_1$ is replaced with $1/m_{\min L}$). It follows that
\begin{eqnarray}
\left| \overrightarrow{e}_\gamma\left(i_n\left((T_kx_k - \lambda x_k,0)_{k=1}^n\right)\right) \right| &=& \frac{1}{m_{\min L}}\left|\overrightarrow{f}\left(i_n\left((T_kx_k - \lambda x_k,0)_{k=1}^n\right)\right)\right|\nonumber\\
&\leq& \frac{1}{m_{\min L}} \left(|\lambda| + \max\|T_k\|\right)\nonumber\\
&\leq& \frac{2}{m_{\min L}} \max\{\max\|T_k\|, |\lambda|\}\label{le grand bleu}
\end{eqnarray}
Combining \eqref{i label this upon thee} with \eqref{le grand bleu} we obtain the desired bound.

Now assume that $\gamma\in\Delta_q^1$. Applying Lemma \ref{analyze back} we obtain that either $\overrightarrow{e}_\gamma\left(i_n\left((T_kx_k - \lambda x_k,0)_{k=1}^n\right)\right) = 0$ or that there is $j\in L$, $1\leq p\leq n$ and $\overrightarrow{b}^*\in\left(\sum_{k=1}^n\oplus\ell_1(\Delta_k)\right)_1$ with $\|\overrightarrow{b}^*\|\leq 1$ so that  $$\overrightarrow{e}_\gamma\left(i_n\left((T_kx_k - \lambda x_k,0)_{k=1}^n\right)\right) = \frac{1}{m_j}\overrightarrow{b}^*\circ P_{[p,n]}\circ i_n\left((T_kx_k - \lambda x_k, 0)_{k=1}^n\right).$$ and therefore we obtain
\begin{eqnarray}
\left|\overrightarrow{e}_\gamma\left(i_n\left((T_kx_k - \lambda x_k,0)_{k=1}^n\right)\right)\right| &\leq& \frac{1}{m_j}\left\|P_{[p,n]}\right\|\left(|\lambda| + \max\|T_k\|\right)\nonumber\\
&\leq&\frac{1}{m_{\min L}}\frac{2m_{\min L}}{m_{\min L} - 2}\left(|\lambda| + \max\|T_k\|\right)\nonumber\\
 &=& \frac{2}{m_{\min L}-2} \left(|\lambda| + \max\|T_k\|\right)\nonumber\\
 &\leq & \frac{4}{m_{\min L}-2} \max\{\max\|T_k\|, |\lambda|\}\label{oh mon dieu}.
\end{eqnarray}
Finally, combining \eqref{i label this upon thee} and \eqref{oh mon dieu} we conclude that
$$|\overrightarrow{e}_\gamma(y)| \leq \left(1 + \frac{4}{m_{\min L}-2}\right)\max\left\{\max_{1\leq k\leq n}\|T_k\|, |\lambda|\right\}$$
which completes the proof.
\end{proof}



\subsection{Well founded trees}

A tree is an ordered space $(\T, \leq)$ such that for every
$t\in \T$ the set $\{s\in\T: s\leq t\}$ is well ordered. An
element $t\in\T$ will be called a node of $\T$ and a minimal node
will be called a root. A subset $\mathcal{S}$ of $\mathcal{T}$ is
called a downwards closed subtree of $\T$ if for every $t\in\mathcal{S}$, the set
$\{s\in\T: s\leq t\}$ is a subset of $\mathcal{S}$. For every node
$t$ we denote by $\scc(t)$ the set of the immediate successors of
$t$ while by $\T_t$ we denote the set $\{s\in\T: t\leq s\}$. Note
that the set $\T_t$ with the induced ordering is also a tree. If
$t$ is a node of $\T$ and $s\in\scc(t)$, then $t$ is called the
immediate predecessor of $s$ and is denoted by $s^-$. For a node
$t$ we define the height of $t$, denoted by $|t|$, to be the order
type of the well ordered set $\{s\in\T: s\leq t\}$. In particular,
if $t$ is a root of $\T$, then $|t| = 0$. A tree is called well
founded if it does not contain any infinite chains. Note that in
this case $|t|$ is finite for every $t\in\T$.

\begin{remark}
From now on, unless stated otherwise, every tree $\T$ will be
assumed to be well founded, have a unique root, denoted by
$\emptyset_{\T}$, and for every non-maximal node $t$ the set
$\scc(t)$ will be assumed to be infinitely countable. Note that
such a tree is either infinitely countable or a
singleton.\label{remark tree type}
\end{remark}

A tree $\T$ is equiped with the compact Hausdorff topology having
the sets $\T_t$, $t\in\T$ as a subbase. Then a
node is an isolated point if and only if the set $\scc(t)$ is
finite, which (following the convention from Remark \ref{remark tree type}) is the case if and only if $t$ is a maximal node.
Moreover, the set of all maximal nodes of $\T$ is dense in $\T$.

For a tree $\T$ the derivative $\T^\prime$ of $\T$ is defined to
be the downwards closed subtree of all non-maximal nodes of $\T$. For an ordinal
number $\alpha$ the derivative of order $\alpha$  of $\T$, which
is also a downwards closed subtree of $\T$, is defined recursively. For $\alpha =
0$ set $\T^0 = \T$. Assuming that for an ordinal number $\alpha$
the derivatives $\T^\beta$ have been defined for all $\beta <
\alpha$, define $\T^\alpha = \cap_{\beta<\alpha}\T^\beta$ if
$\alpha$ is a limit ordinal number and $\T^\alpha =
(\T^\beta)^\prime$ if $\alpha$ is a successor ordinal number with
$\alpha = \beta + 1$. The rank of the tree $\T$, denoted by
$\rho(T)$, is defined to be the smallest ordinal number $\alpha$
satisfying $\T^\alpha = \T^{\alpha + 1} = \varnothing$. Note that
$\rho(\T)$ is a countable ordinal number.

For a node $t$ the rank of $t$, with respect to the tree $\T$,
denoted by $\rho_{_{\T}}(t)$, is also defined to be the supremum
of all ordinal number $\alpha$ such that $t\in\T^\alpha$. Observe
that the supremum is actually obtained, i.e. if $\rho_{_{\T}}(t) =
\al$ then $t\in\T^{\al}$. The function $\rho_{_{\T}}$ satisfies
the following properties: $\rho_{_{\T}}(t) = 0$ if and only if $t$
is a maximal node and otherwise $\rho_{_{\T}}(t) =
\sup\{\rho_{_{\T}}(s) + 1:\; t<s\} = \sup\{\rho_{_{\T}}(s) + 1:\;
s\in\scc(t)\}$. Moreover $\rho(\T) = \sup\{\rho_{_{\T}}(t) + 1:\;
t\in\T\} = \rho_{_{\T}}(\emptyset_{\T}) + 1$. Furthermore, if $s$
is a node of $\T$ and we consider the tree $\T_s$ with the induced
ordering, then $\rho_{_{\T}}(t)  = \rho_{_{\T_s}}(t)$ for every
$t\in\T_s$.

\begin{remark}
It is more convenient to use the rank $\rho_{_{\T}}(\es_{\T})$ of
the root of a tree $\T$ instead of the rank $\rho(\T)$ of the tree
itself. Therefore, unless stated otherwise, by rank of a tree $\T$
we shall mean the ordinal number $o(\T) = \rho_{_{\T}}(\es_{\T})$.
Under this notation, we have that $o(\T) = 0$ if and only if $\T$
is a singleton and otherwise $o(\T) = \sup\{ o(\T_s) + 1:\;
s\in\T$ with $s\neq\es_{\T}\}$.\label{remark rank of tree}
\end{remark}

\section{The spaces $X_{(\T,L,\e)}$}\label{the definition}

For every tree $\T$, infinite subset of the natural numbers $L$
and positive real number $\e$ we define a Banach space
$X_{(\T,L,\e)}$. Let us for now forget the parameter $\e$, which represents the declination of the space's dual from $\ell_1$, and concentrate on the set $L$, which plays a significant role in proving the properties of the space, so let us for the moment denote the spaces in the form $X_{(\T,L)}$.  The purpose of this set $L$ lies in the fact that, as already mentioned, the Argyros-Haydon space is constructed using a sequence of pairs of natural numbers $(m_j,n_j)_j$, called weights and using a subsequence $(m_j,n_j)_{j\in L}$ of these weights one can define the space $X_{AH(L)}$. As it is proved in \cite{AH}, if the intersection of two sets $L,M$ is finite, then every operator from $X_{AH(L)}$ to $X_{AH(M)}$ is compact. Returning now to the recursive definition of the spaces (which we for now denote by $X_{(\T,L)}$), in the basic recursive step the space $X_{(\T,L)}$ is the space $X_{AH(L)}$, while in the general case the space $X_{(\T,L)}$ is the direct sum $(\sum\oplus X_{(\T_n,L_n)})_{AH(L_0)}$, where $(L_n)_{n}$ defines a partition of $L$. This method of construction endows these spaces with the following properties: every operator defined on $X_{(\T,L)}$ is a multiple of the identity plus a horizontally compact operator (see definition \ref{zontirohaly mopact}) and if $\s$ is a tree while $M$ is such that its intersection with $L$ is finite, then every operator $T:X_{(\T,L)}\rightarrow X_{(\s,M)}$ is compact. These facts are the main tools used to prove Theorem \ref{a}  in the introduction.

\begin{definition}
By transfinite recursion on the order $o(\T)$ of a tree $\T$, we
define the spaces $X_{(\T,L,\e)}$ for every $L$ infinite subset of
the natural numbers and $\e$ positive real number.
\begin{enumerate}

\item[(i)]
Let $\T$ be a tree with $o(\T) = 0$ (i.e. $\T =
\{\es_{\T}\}$), $L$ be an infinite subset of the natural numbers
and $\e$ be a positive real number. Let $\de >0$ with  $(1+\de)^2 < 1 + \e$. Choose $L^\prime$ an infinite subset of $L$ such that $4/(\min L^\prime - 2) < \de$,
and define $X_{(\T,L,\e)} = X_{AH(L^\prime)}$.

\item[(ii)]
Let $\T$ be a tree with $0<\al = o(\T)$, $L$ be an infinite
subset of the natural numbers and $\e$ be a positive real number.
Assume that for every tree $\s$ with $o(\s) < \al$, for every
infinite subset of the natural numbers $M$ and positive real
number $\e^\prime$ the space $X_{(\s,M,\e^\prime)}$ has been
defined.  Choose $\{s_n: n\inn\}$ an enumeration of the set
$\scc(\es_{\T})$, $\de >0$ with  $(1+\de)^2 < 1 + \e$, $L^\prime$ an infinite subset of $L$ with
$4/(\min L^\prime - 2) < \de$  and a partition of $L^\prime$ into infinite sets $(L_n)_{n=0}^\infty$. Define $X_{s_n} =
X_{(\T_{s_n},L_n,\de)}$ and
    \begin{equation*}
    X_{(\T,L,\e)} = \left(\sum_{n=1}^{\infty}\oplus X_{s_n} \right)_{AH(L_0)}.
    \end{equation*}

\end{enumerate}\label{def xtle}
\end{definition}

\begin{remark}\label{hi saturated}
Proposition \ref{dualZAH}(ii) and a transfinite induction on $o(\T)$ yield that the space $\X$ is HI-saturated, for all $\T, L, \e$ as in Definition \ref{def xtle}.
\end{remark}

\begin{proposition}
Let $\T, L, \e$ be as in Definition \ref{def xtle}. Then
$X_{(\T,L,\e)}^*$ is $(1+\e)$-isomorphic to $\ell_1$.\label{l1
predual}
\end{proposition}

\begin{proof}
We use transfinite induction on the order $o(\T)$ of a tree $\T$. In the case that $o(\T)=0$ the result follows by Remark \ref{constantAH}. We now assume that $o(\T)=\alpha>0$ and that for every tree $\s$ with $o(\s)<\alpha$, every infinite subset $M$ of $\N$ and every $\e^\prime>0$ the dual space $X_{(\s,M,\e^\prime)}^*$ is $(1+\e^\prime)$-isomorphic to $\ell_1$. Let $\{s_n: n\inn\}$, $L^{\prime}$, $(L_n)_{n=0}^\infty$, $\delta>0$
 and $X_{s_n}$ as in Definition \ref{def xtle} (2) such that \[X_{(\T,L,\e)} = \left(\sum_{n=1}^{\infty}\oplus X_{s_n} \right)_{AH(L_0)}.\] Observe that since $L_0\subset L\prime$ it follows that $2/(\min L_0-2) < \delta$ and hence by Proposition \ref{dualZAH} and Remark \ref{constant} there exists $(\Delta_n)_{n\in\N}$ pairwise disjoint subsets of $\N$ such that \[X_{(\T,L,\e)}^*\simeq^{1+\delta} =\left(\sum_{n=1}^{\infty}\oplus\left(X_{s_n}^*\oplus\ell_1(\Delta_n)\right)_1\right)_1.\]
Since $X_{s_n}=X_{(\T_{s_n},L_n,\de)}$ and $o(\T_{s_n})<\alpha$ for every $n\in\N$, applying our inductive assumption we conclude that $X_{s_n}^*\simeq^{1+\delta}\ell_1$. By the choice of $\delta$ the result follows.
\end{proof}

\begin{remark}\label{scriptLinfty}
The above in conjunction with the principle of local reflexivity \cite{LR} yields that the space $X_{(\T,L,\e)}$ is a $\mathcal{L}_{\infty,(1+\e')}$-space for every $\e'>\e$.
\end{remark}

At this point we shall make a few observations that follow from Definition \ref{def xtle} and the discussion made in Section \ref{preliminaries}. First, we mention that each space $X_{(\T,L,\e)}$ is associated with a sequence $(\Delta_n)_{n\in\N}$ that is involved in its construction and we denote the union $\cup_n\Delta_n$ by $\Gamma(\T,L,\e)$. In particular, if $o(T)=0$ the space $X_{(\T,L,\e)}$ has an FDD $(M_n)_{n\in\N}$ such that each $M_n$ is isometric to $\ell_{\infty}(\Delta_n)$, while in the case that $o(\T)>0$ and $\{s_n: n\inn\}$ is the enumeration of the set $\scc(\es_{\T})$ as in the definition of the space $X_{(\T,L,\e)}$, the space admits a Schauder Decomposition $(Z_n)_{n\in\N}$ such that each $Z_n$ is isometric to $(X_{s_n}\oplus\ell_{\infty}(\Delta_n))_{\infty}$.

Using the already introduced terminology of Section \ref{preliminaries}, for the sequel we denote by $P_n$ the projections associated with the FDD or the Schauder Decomposition of the space $X_{(\T,L,\e)}$, where the image $Im P_n=Z_n$ by Lemma \ref{isometryXn} is isometric to $(X_{s_n}\oplus\ell_{\infty}(\Delta_n))_{\infty}$. Also, for $n\in\N$, we denote by $j_n:X_{s_n}\rightarrow \left(\sum_{k=1}^n\oplus(X_{s_n}\oplus\ell_{\infty}(\Delta_n))_{\infty}\right)_\infty$ the natural embedding, as well as by $\pi_n:\left(\sum_{k=1}^n\oplus(X_{s_n}\oplus\ell_{\infty}(\Delta_n))_{\infty}\right)_\infty\rightarrow X_{s_n}$ the natural coordinate projection and by $R_{[1,n]}:\mathcal{Z}\rightarrow \left(\sum_{k=1}^n\oplus(X_{s_n}\oplus\ell_{\infty}(\Delta_n))_{\infty}\right)_\infty$ the natural restriction mappings. We define the projections $I_n:\mathcal{Z}\to \mathcal{Z}$ as
$I_n=i_n\circ j_n\circ \pi_n\circ R_{[1,n]}$. Lemma \ref{constant projectionIn} yields that $\|I_n\|=1$ for every $n\in\N$ and by Corollary \ref{isometryXn} we obtain that the image $Im I_n$ is isometric to $X_{s_n}$.

In a similar manner as above, for each $t\in\T$ non maximal, following the notation in Proposition \ref{tree xtle}(iv) we denote by $P_n^t$ the projections defined upon $X_t$ such that for each $n$ the image $Im P_n^t$ is isometric to $(X_{s_n}\oplus\ell_{\infty}(\Delta_n^t))_{\infty}$ and $\cup_n\Delta_n^t=\Gamma(\T_t,L_t,\delta_t)$. Note that by the above $P_n=P_n^{\es_{\T}}$. Moreover, for $I=(n,m]$ interval of natural numbers and $t\in\T$ we define $P_I^t=\sum_{i=n+1}^mP_i^t$.

Finally, the notion of a block sequence in $X_{(\T,L,\e)}$ is defined as in Section \ref{preliminaries} using the FDD (if $o(T)=0$) or the Schauder Decomposition (if $o(T)\neq0$). In the second case we call the block sequence horizontally block.

\begin{proposition}
Let $\T, L, \e$ be as in Definition \ref{def xtle}. Then there
exist $(L_s)_{s\in\T}$ infinite subsets of the natural numbers,
$(\e_s)_{s\in\T}$ positive real numbers, $(I_s)_{s\in\T}$ norm
one projections defined on $X_{(\T,L,\e)}$ and $(X_s)_{s\in\T}$
infinitely dimensional subspaces of $X_{(\T,L,\e)}$ such that the
following are satisfied:
\begin{enumerate}

\item[(i)] The sets $L_{\es_{\T}}$ and $L$ are equal, if $s,t$ are
nodes in $\T$ then  if $t$ and $s$ are incomparable we have that
$L_s\cap L_t = \varnothing$ while if $s\leq t$ then we have that
$L_t\subset L_s$.

\item[(ii)] The projection $I_{\es_{\T}}$ is the identity map, if
$s,t$ are nodes in $\T$ then if $t$ and $s$ are incomparable we
have that $\im I_s \subset \ker I_t$ and if $s\leq t$ then
$I_s\circ I_t = I_t\circ I_s = I_t$.

\item[(iii)]
The image of the operator $I_s$ is the space $X_s$ and
$X_s$ is isometric to $X_{(T_s,L_s,\e_s)}$.

\item[(iv)]
If $t$ is a non-maximal node of $\T$, then there exists
an enumeration $\{s_n:\; n\inn\}$ of $\scc(t)$ and $L_t^0$ an
infinite subset of $L_t\setminus \cup_{s\in\scc(t)} L_s$ such that
the following is satisfied:
\begin{equation*}
X_t = \left(\sum_{n=1}^\infty\oplus X_{s_n}\right)_{AH(L_t^0)}.
\end{equation*}

\end{enumerate}\label{tree xtle}
\end{proposition}

\begin{proof}
We use induction on the order of the tree $\T$, $o(\T)$. In the case that $o(\T)=0$, let $L^{\prime}$ be such that $X_{(\T,L,\e)} = X_{AH(L^\prime)}$. We set $L_{\es_{\T}}=L$, $I_{\es_{\T}}=I$, where $I$ denotes the identity map on $X_{AH(L^\prime)}$ and $\de_{\es_{\T}}=\e$. Assume that $o(\T)=\alpha>0$ and that the statement has been proved for every $\s, M, \delta$ such that $o(\s)<\alpha$, $M\subseteq\N$ infinite and $\delta>0$.

Let $\{s_n: n\inn\}$, $(L_n)_{n=0}^\infty$, $\delta>0$
 and $X_{s_n}$ as in Definition \ref{def xtle} (2) such that \[X_{(\T,L,\e)} = \left(\sum_{n=1}^{\infty}\oplus X_{s_n} \right)_{AH(L_0)}.\] Let also $I_n:X_{(\T,L,\e)}\to X_{(\T,L,\e)}$ be the norm 1 projections for every $n\in\N$ such that $Im I_n$ is isometric to $X_{s_n}$. We recall that $X_{s_n}=X_{(\T_{s_n},L_n,\de)}$ and since $o(\T_{s_n})<\alpha$ applying our inductive assumption for every $n\in\N$ we obtain $(L_s)_{s\in\T_{s_n}}$ infinite subsets of the natural numbers,
$(\e_s)_{s\in\T_{s_n}}$ positive real numbers, $(I_s)_{s\in\T_{s_n}}$ norm
one projections defined on $X_{(\T_{s_n},L_n,\de)}$ satisfying the conditions (i)-(iv) above.

We set $L_{\es_{\T}}=L$, $I_{\es_{\T}}=I$, where $I$ denotes the identity map on $X_{(\T,L,\e)}$, $\e_{\es_{\T}}=\e$. For $t\in\T$ with $t\neq\es_{\T}$ note that there exists $n_0\in\N$ such that $t\in\T_{s_{n_0}}$, hence $L_t$, $\e_t$ and $I_t:X_{s_{n_0}}\to X_{s_{n_0}}$ where $Im I_t$ is isometric to $X_t$, have already been defined by our inductive assumption.

We extend $I_t:X_{(\T,L,\e)}\to X_{(\T,L,\e)}$ as $I_t=I_t\circ I_{n_0}$ and we set $X_t=I_t[X_{(\T,L,\e)}]$. It is not hard to check that $(L_s)_{s\in\T}$,
$(\e_s)_{s\in\T}$, $(I_s)_{s\in\T}$ and $(X_s)_{s\in\T}$ satisfy the thesis.
\end{proof}

\section{Operators on the spaces $X_{(\T,L,\e)}$}\label{the operators}
We study the properties of the operators defined on the spaces $\X$. The purpose is to show that every bounded linear operator is approximated by a sequence of operators, each one of which is a linear combination of the projections $I_s, s\in\T$ plus a compact operator.
We start by giving the definition of Rapidly Increasing Sequences (RIS) in the spaces $X_{(\T,L,\e)}$.

\begin{definition}
Let $L$ be an infinite subset of $\N$, $(X_n)_n$ be a sequence of separable Banach spaces and $X = \left(\sum_n\oplus X_n\right)_{AH(L)}$. We say that a block sequence (respectively a horizontally block sequence) $(z_k)_{k\in\N}$ in $X_{AH(L)}$ (respectively in $X$) is a $C$-rapidly increasing sequence (RIS) if there exists a constant $C>0$ and a strictly increasing sequence $(j_k)_{k\in\N}$ in $L$ such that
\begin{enumerate}
\item[(i)] $\|z_k\|\leq C$ for all $k\in\N$
\item[(ii)] $j_{k+1}>\max\ran z_k$
\item[(iii)] $|\overrightarrow{e}_{\gamma}^*(z_k)|\leq \frac{C}{m_i}$ for every $\gamma\in\Gamma$ with $w(\gamma)<m_{j_k}$.
\end{enumerate}
\end{definition}
Note that in the case where of $X_{AH(L)}$ the definition of a $C$-RIS essentially coincides with the corresponding one presented in \cite{AH}. Moreover, the existence of C-RIS in $X = \left(\sum_n\oplus X_n\right)_{AH(L)}$ is proved in a similar manner as it is described in \cite{Z} and makes use of Proposition \ref{upper AH sum} as well as an analogue of \cite[Lemma 8.4]{AH}.

The next result follows readily from \cite[Proposition 5.4]{AH} and \cite[Proposition 5.12]{Z}.

\begin{proposition}\label{basiceq}
Let $L$ be an
infinite subset of the natural numbers and $X$ be either $X_{AH(L)}$ or $X = \left(\sum_{n=1}^\infty\oplus
X_n\right)_{AH(L)}$, where $(X_n)_n$ is a sequence of separable Banach spaces. Let also and $(z_k)_k$ be a C-RIS in $X$ and $j_0\in\N$. Then, \[\left\|n_{j_0}^{-1}\sum_{k=1}^{n_j}z_k\right\|\leq\begin{cases} \frac{10C}{m_{j_0}}\ \ \text{if }\ j_0\in L\ \\ \frac{10C}{m_{j_0}^2}, \ \ \text{if} \ \ j_0\notin L.\end{cases}\]
\end{proposition}

The following is proved in \cite[Proposition 5.14]{Z}
\begin{lemma}
Let $L$ be an
infinite subset of the natural numbers, $(X_n)_n$ be a sequence of separable Banach spaces and $X = \left(\sum_{n=1}^\infty\oplus
X_n\right)_{AH(L)}$. Let also $Y$ be a Banach space, $T: X \rightarrow Y$ be a bounded linear operator and
assume that there exists a seminormalized horizontally block
sequence $(x_k)_k$ in $X$ such that $\lim\sup_k\|Tx_k\|>0$. Then
there exists a RIS $(y_k)_k$ in $X$ such that
$\lim\sup_k\|Ty_k\|>0$.\label{not hc}
\end{lemma}

We adapt the following definition given in \cite{Z}.

\begin{definition}\label{zontirohaly mopact}
Let $(X_n)_n$ be a sequence of separable Banach spaces, $L$ be an infinite subset of the natural numbers, $X = \left(\sum_{n=1}^\infty\oplus
X_n\right)_{AH(L)}$ and $Y$ be a Banach space. We say that an operator $K: X \to Y$ is horizontally compact if for every $\delta>0$ there exists $n_0\in\N$ such that $\|K-K\circ P_{[1,n_0]}\|<\delta$. Equivalently, $K$ is horizontally compact if $\lim_k\|K(x_k)\| = 0$ for every horizontally block sequence $(x_k)_k$ in $X$.

\end{definition}

\begin{proposition}\label{block to RIS}
Let $(X_n)_n$ be a sequence of separable Banach spaces, $L$ be an
infinite subset of $\N$ and $Y$ be a Banach space. The following hold:
\begin{enumerate}
\item[(i)] If a bounded linear operator $T:X_{AH(L)}\rightarrow Y$ is not compact, then there exists a RIS  $(x_k)_k$ in $X_{AH(L)}$ such that $\lim\sup_k\|Tx_k\|>0$.
\item[(ii)] If a bounded linear operator $T:\left(\sum_{n=1}^\infty\oplus
X_n\right)_{AH(L)}\rightarrow Y$ is not horizontally compact, then there exists a RIS $(x_k)_k$ in $\left(\sum_{n=1}^\infty\oplus
X_n\right)_{AH(L)}$ such that $\lim\sup_k\|Tx_k\|>0$.
\end{enumerate}
\end{proposition}

\begin{proof}
For (1) we first observe that since $T$ is not compact there exists $\delta>0$ and a bounded block sequence $(x_k)_k$ in $X_{AH(L)}$ such that $\|T(x_k)\|\geq\delta$. By \cite[Proposition 5.11]{AH} the result follows. For (2) assume towards a contradiction that $\lim\sup_k\|Tx_k\|=0$ for every RIS sequence $(x_k)_k$
in $X$. By Lemma \ref{not hc} again we conclude that $\lim\sup_k\|Tx_k\|=0$ for every bounded horizontally block sequence in $X$. It follows that $T$ is horizontally compact yielding a contradiction.
\end{proof}

\begin{lemma}
Let $(X_k)_k$, $(Y_k)_k$ be sequences of separable Banach spaces
and $L$, $M$ be infinite subsets of $\N$ such that
$L\cap M$ is finite. Let moreover $(x_k)_k$ be a C-RIS in
$\left(\sum_{n=1}^\infty\oplus X_n\right)_{AH(L)}$ and $(y_k)_k$
be a seminormalized horizontally block sequence in
$\left(\sum_{n=1}^\infty\oplus Y_n\right)_{AH(M)}$. Then $(x_k)_k$
does not dominate $(y_k)_k$, i.e. the map $x_k\rightarrow y_k$
does not extend to a bounded linear operator. \label{no domination}
\end{lemma}

\begin{proof}
By Proposition \ref{upper AH sum} there exists $\tilde{M}$ subset of $M$ with the property that for every $j\in \tilde{M}$ passing to a subsequence we have that $\|\sum_{k=1}^{n_j}y_k\|\geq\frac{1}{2m_j}\sum_{i=1}^{n_j}\|y_k\|\geq\frac{1}{m_{\min M}}\frac{n_j}{4m_j}$. Moreover, by Proposition \ref{basiceq} for $j\notin L$ we obtain that $\|\sum_{k=1}^{n_j}x_k\|\leq\frac{10C n_j}{m_j^2}$. Assume that there exists a constant $c>0$ such that $\|\sum_{k=1}^{n}a_kx_k\|\geq c\|\sum_{k=1}^{n}a_ky_k\|$ for every $n\in\N$ and every sequence of scalars $(a_k)_k$. Since $M\cap L$ is finite, we can choose $j\in \tilde{M}$, $j\notin L$ such that $m_j>m_{\min M}\cdot\frac{40C}{c}$. Combining the above and passing to subsequences we conclude that \[\frac{1}{m_{\min M}}\frac{cn_j}{4m_j}\leq c\left\|\sum_{k=1}^{n_j}y_k\right\|\leq\left\|\sum_{k=1}^{n_j}x_k\right\|\leq\frac{10C n_j}{m_j^2}.\] The choice of $j$ yields a contradiction.
\end{proof}

\begin{remark}\label{ntomata} In a similar manner, for $L, M$, $(Y_k)_k$ as above, a C-RIS in $X_{AH(L)}$ cannot dominate a seminormalized horizontally block sequence in the space $\left(\sum_{n=1}^\infty\oplus Y_n\right)_{AH(M)}$.
\end{remark}

\begin{lemma}
Let $(X_n)_n$ be a sequence of separable Banach spaces
and $M$ be an infinite subset of the natural numbers. Let also
$\T, L, \e$ be as in Definition \ref{def xtle} such that $L\cap M$
is finite. Then every bounded linear operator $T:
\left(\sum_{n=1}^\infty\oplus X_n\right)_{AH(M)}\rightarrow
X_{(\T,L,\e)}$ is horizontally compact.\label{for dist}
\end{lemma}

\begin{proof}
We use transfinite induction on the rank of the tree $o(\T)$. Suppose that $o(\T)=0$ and assume towards a contradiction that $T$ is not horizontally compact. By Proposition \ref{block to RIS} there exists a RIS sequence $(x_k)_k$ and $\delta>0$ such that $\|T(x_k)\|\geq\delta$. We may assume that $(x_k)_k$ is normalized and note that $(T(x_k))_k$ is weakly null.

Since $X_{(\T,L,\e)}=X_{AH(L^{\prime})}$ for some $L^{\prime}\subseteq L$ we may assume that $(T(x_k))_k$ is a block sequence and by Proposition \ref{upper} we obtain $\tilde{L}\subseteq L^\prime$ infinite such that for every $j\in \tilde{L}$, $\|\sum_{k=1}^{n_j}T(x_k)\|\geq \frac{1}{m_{\min L'}}\frac{\delta n_j}{4m_j}$, passing to a subsequence. Let $j\in\tilde{L}$ such that $j\notin M$ and $m_j>m_{\min L'}\cdot\frac{10\|T\|}{\delta}$. Proposition \ref{basiceq} yields that $\|\sum_{k=1}^{n_j}x_k\|\leq\frac{10 n_j}{m_j^2}$ yielding a contradiction by the choice of $j$.

Suppose now that $o(\T)=\alpha>0$ and assume that for every $\s, M^\prime, \e^\prime$  as in Definition \ref{def xtle} such that $o(\s)<\alpha$, $M^\prime\cap M$ is finite,  every bounded linear operator $T:
\left(\sum_{n=1}^\infty\oplus X_n\right)_{AH(M)}\rightarrow
X_{(\s,M^\prime,\e^\prime)}$ is horizontally compact. We will prove that the same holds for $T:
\left(\sum_{n=1}^\infty\oplus X_n\right)_{AH(M)}\rightarrow
X_{(\T,L,\e)}$ and let $\{s_n: n\inn\}$, $(L_n)_{n=0}^\infty$, $\delta>0$
as in Definition \ref{def xtle} (2) such that $X_{(\T,L,\e)} = \left(\sum_{n=1}^{\infty}\oplus X_{s_n} \right)_{AH(L_0)}$ where  $X_{s_n}=X_{(\T_{s_n},L_n,\de)}$. Assume towards the contradiction that this is not the case and by Proposition \ref{block to RIS} let $(x_k)_k$ be a normalized RIS sequence  and $\delta>0$ such that $\|T(x_k)\|\geq\delta$, for every $k\in\N$. For $m\inn$ we consider the operator
\[R_{[1,m]}\circ T:\left(\sum_{n=1}^\infty\oplus X_n\right)_{AH(M)}\rightarrow \left(\sum_{n=1}^m\oplus\left(X_{s_n}\oplus\ell_{\infty}(\Delta_n)\right)_\infty\right)_{\infty},\]
where $\cup_n\Delta_n=\Gamma(\T,L,\e)$. Using our inductive assumption we deduce that the operator $R_{[1,m]}\circ T$ is horizontally compact, hence $\lim_k\|P_{[1,m]}(T(x_k))\| = \lim_k \|i_m\circ R_{[1,m]}(T(x_k))\| = 0$.

By a sliding hump argument, it follows that $(T(x_k))_k$ is equivalent to a horizontally block sequence in $X_{(\T,L,\e)}$ and, since $T$ is bounded, by Lemma \ref{no domination} we arrive to a contradiction.
\end{proof}

\begin{remark}\label{AHcompact}
For $M$, $\T, L, \e$ as above we obtain that every bounded linear operator $T:X_{AH(M)}\rightarrow X_{(\T,L,\e)}$ is compact.
\end{remark}

The following is similar to \cite[Theorem 10.4]{AH}.

\begin{proposition}
Let $\T,L,\e$ and $\s, M, \de$ be as in Definition \ref{def xtle}
such that $L\cap M$ is finite. Then every bounded linear operator
$T: X_{(\T,L,\e)}\rightarrow X_{(\s,M,\de)}$ is
compact.\label{disjoint weights compact}
\end{proposition}

\begin{proof}
We use a transfinite induction on $o(\T)$. In the case that $o(\T)$ is zero, let $L^\prime$ such that $X_{(\T, L, \e)} = X_{AH(L^\prime)}$ and observe that the statement follows by Remark \ref{AHcompact}.

Assume now that $o(\T)=\alpha$  and suppose that for every tree $\T^\prime$ with $o(\T^\prime) <\alpha$,  every operator $T: X_{(\T^\prime,L^\prime,\e)}\rightarrow X_{(\s,M,\de)}$ is compact for every $\s, M, \de$ that satisfy Definition \ref{def xtle}. Fix $\s, M, \de$ and let also $\{s_n: n\inn\}$, $(L_n)_{n=0}^\infty$, $\delta>0$
 as in Definition \ref{def xtle} (2) such that $X_{(\T,L,\e)} = \left(\sum_{n=1}^{\infty}\oplus X_{s_n} \right)_{AH(L_0)}$ where $X_{s_n}=X_{(\T_{s_n},L_n,\de)}$. In order to show that $T: X_{(\T,L,\e)}\rightarrow X_{(\s,M,\de)}$ is
compact, we first need to observe that by Lemma \ref{for dist} the operator $T$ is horizontally compact. Hence, $\lim_k\|T-TP_{[1,k]}\| = 0$.

Let $I_{s_n}:X_{(\T,L,\e)}\to X_{s_n}$ be the projections defined in Proposition \ref{tree xtle}. By Remark \ref{difference compact} the operator $T\circ P_{[1,k]}-\sum_{n=1}^kT\circ I_{s_n}$ is compact. Consider the bounded operators $T\circ I_{s_n}:X_{s_n}\rightarrow X_{(\s,M,\de)}$. Since $o(\T_{s_n})<\alpha$  applying our inductive assumption we conclude that $T\circ I_{s_n}$ is compact for every $n=1,\ldots,k$. It follows that $T\circ P_{[1,k]}$ is compact and therefore $T = \lim_k T\circ P_{[1,k]}$ is compact as well.
\end{proof}

The statement of the next Lemma is proved as in \cite[Lemma 7.7]{Z} using Proposition \ref{disjoint weights compact}.

\begin{lemma}\label{distance}
Let $\T,L,\e$ be as in Definition \ref{def xtle}, $(x_k)_k$ be a
RIS in $X_{(\T,L,\e)}$ and $T:X_{(\T,L,\e)}\rightarrow
X_{(\T,L,\e)}$ be a bounded linear operator. Then we have that
$\lim_k\dist(Tx_k, \mathbb{R}x_k) = 0$.\label{dist on ris}
\end{lemma}

The next proposition shares similar arguments as in \cite[Theorem 7.4]{AH} and \cite[Proposition 7.8]{Z}.

\begin{proposition}
Let $\T,L,\e$ be as in Definition \ref{def xtle} and
$T:X_{(\T,L,\e)}\rightarrow X_{(\T,L,\e)}$ be a bounded linear
operator. Then there exists a real number $\la$ such that the
operator $\la I - T$ is horizontally compact.\label{scalar plus
hc}
\end{proposition}

\begin{lemma}
Let $\T,L,\e$ be as in Definition \ref{def xtle} and let also $s$
be a node of $\T$. Then for every bounded linear operator $T:
X_{(\T,L,\e)}\rightarrow X_{(\T,L,\e)}$, we have that the operator
$T\circ I_s - I_s\circ T \circ I_s$ is a compact
one.\label{significant part of T}
\end{lemma}

\begin{proof}
We prove this lemma using transfinite induction on the rank of
$\T$. If $o(\T) = 0$ then $\T = \{\es_{\T}\}$ and therefore $s =
\es_{\T}$, i.e. $I_s$ is the identity map. We conclude that
$T\circ I_s - I_s\circ T \circ I_s$ is the zero operator, which is
compact.

Assume now that $\al$ is a countable cardinal number such that the
statement holds for every $\s$, $M$, $\de$ as in Definition \ref{def xtle} with $o(\s) < \al$ and let
$\T$ be a tree with $o(\T) = \al$. Assume that $X_{(\T,L,\e)} =
\left(\sum_{n=1}^{\infty}\oplus X_{s_n} \right)_{AH(L_0)}$, where
$\{s_n:\;n\inn\}$ is an enumeration of $\scc(\es_{\T})$ and
$L_t\cap L_0 = \varnothing$ for every $t\in\T$ with
$t\neq\es_{\T}$.

We shall use transfinite induction once more, this time on the
rank $\rho_{\T}(s)$ of the node $t$. Assume that $\rho(s) = 0$,
i.e. $X_s = X_{AH(L_s)}$. Let $S: X_s\rightarrow X_{(\T,L,\e)}$
be the restriction of $T\circ I_s - I_s\circ T \circ I_s$ onto $X_s$. As $I_s$ is a projection onto $X_s$, it is
enough to show that $S$ is compact. Assume that it is not, then by
Proposition \ref{block to RIS}(1), there exists a RIS $(x_k)_k$ in $X_s$ and a
positive real number $\theta$ such that $\|Sx_k\| > \theta$ for all
$k\inn$.

Let $n_0$ denote the unique natural number such that
$s\in\T_{s_{n_0}}$ and let $P_{n}$ denote the natural projections
on the components $Z_n = i_n\left[\left(X_{s_n}\oplus\ell_\infty(\Delta_n)\right)_\infty\right]$. Recall
that from Remark \ref{difference compact} $P_{n} - I_{s_n}$ is a finite rank operator and therefore it is
compact. For $n\neq n_0$, since $L_{s_n}\cap L_{s_{n_0}} = \varnothing$ and $L_s\subset L_{s_{n_0}}$,
Proposition \ref{disjoint weights compact} yields that $P_{n}\circ
S$ is compact. Moreover, for $n = n_0$ the inductive assumption
implies that the map $P_{n}\circ S$ is compact and therefore we
have that $\lim_kP_{n}\circ Sx_k = 0$ for all $n\inn$. We conclude
that the sequence $(Sx_k)_k$ has a subsequence equivalent to a
horizontally block sequence of $X_{(\T,L,\e)}$. Since $L_0\cap L_s
= \varnothing$, Lemma \ref{no domination} yields a contradiction.

Assume now that $0<\be\leq o(\T)$ is an ordinal number such that
the statement holds for every $s\in\T$ with $\rho_{\T}(s) < \be$
and let $s$ be a node of $\T$ with $\rho_{\T}(s) = \be$. If $s$ is
the root of the tree, then by the fact that $I_{\es_{\T}}$ is the
identity map, one can easily deduce the desired result. It is
therefore sufficient to check the case in which $\rho_{\T}(s) <
\rho_{\T}(\es_{\T}) = o(\T) = \al$.

Since $s$ is a non maximal node, by Proposition \ref{tree xtle}
there exists an enumeration $\{t_n:\; n\inn\}$ of $\scc(s)$ and
$L_s^0$ an infinite subset of $L_s\setminus \cup_{t\in\scc(s)}
L_t$ such that $X_s = \left(\sum_{n=1}^\infty\oplus
X_{t_n}\right)_{AH(L_s^0)}$.

Setting $S = I_{s}\circ T\circ I_{s}$, since $o(\T_{s}) < \al$,
the inductive assumption yields that the operators $S\circ I_{t_n}
- I_{t_n}\circ S\circ I_{t_n}$ are compact. In other words, the
operators $I_s\circ T\circ I_{t_n} - I_{t_n}\circ T\circ I_{t_n}$
are compact for all $n\inn$. Moreover, since $\rho_{\T}(t_n) <
\be$ for all $n\inn$, the second inductive assumption yields that
the operators $T\circ I_{t_n} - I_{t_n}\circ T\circ I_{t_n}$ are
compact for all $n\inn$. We conclude that the operators $T\circ
I_{t_n} - I_s\circ T\circ I_{t_n}$ are compact for all $n\inn$.

For $n\inn$ we recall that $P^s_{n}$ denotes the natural projections defined
on $X_s$ onto the component
$X_{t_n}\oplus\ell_{\infty}(\Delta_n^s)$ and we denote by $P_{t_n}$
the operators ${P_{n}^s}\circ I_s$. As before, the operators
$I_{t_n} - P_{t_n}$ are compact, which yields the following:
\begin{equation}
\text{For every}\;n\inn\;\text{the operator}\; T\circ P_{t_n} -
I_s\circ T\circ P_{t_n}\;\text{is compact}.\label{significant part
of T1}
\end{equation}
Observe moreover that $I_s = SOT-\sum_{i=1}^\infty P_{t_i}$ and
hence the following holds:
\begin{equation}
T\circ I_s - I_s\circ T\circ I_s =
SOT-\sum_{i=1}^\infty\left(T\circ P_{t_i} - I_s\circ T\circ
P_{t_i}\right).\label{significant part of T2}
\end{equation}
To conclude that the operator $T\circ I_s - I_s\circ T\circ I_s$
is compact, \eqref{significant part of T1} implies that it is
enough to show that the series on the righthand side of
\eqref{significant part of T2} converges in operator norm. In
other words, it is sufficient to show that setting
$R:X_s\rightarrow X_{(\T,L,\e)}$ to be the restriction of $T\circ I_s - I_s\circ
T\circ I_s$ onto $X_s$, then $R$ is horizontally compact.

Towards a contradiction, assume that this is not the case. Lemma
\ref{not hc} yields that there exists a RIS $(x_k)_k$ in $X_s$ and
a positive real number $\theta$ with $\|Rx_k\| > \theta$ for all $k\inn$.
Arguing exactly as in the case $\rho_{\T}(s) = 0$, we conclude
that the sequence $(Rx_k)_k$ has a subsequence equivalent to a
horizontally block sequence of $X_{(\T,L,\e)}$. Since $L_0\cap L_s
= \varnothing$, once more Lemma \ref{no domination} yields a
contradiction.
\end{proof}

\begin{lemma}
Let $\T,L,\e$ be as in Definition \ref{def xtle} with $o(\T) > 0$.
Let also $s$ be a non maximal node of $\T$ and let $\{s_n:\; n\inn\}$ be the enumeration of $\scc(s)$
and $L_s^0$ be an infinite subset of $L$ provided by Proposition \ref{tree xtle} (iv). Let also $T:\X\rightarrow\X$ be a bounded linear
operator. Then there exists a unique real number $\lambda_s$ and a sequence of compact operators
$(C_n)_n$ defined on $\X$, such that if $T_s = I_s\circ T\circ I_s$, then $|\la_s|\leq \|T_s\|$ and the following is satisfied:
\begin{equation}\label{i am fifteen euros short}
\lim_n\left\|T_s -\left(\la_s I_s + \sum_{i=1}^n\left(I_{s_i}\circ (T_s - \la_s I_s) \circ I_{s_i}\right)\right) - C_n\right\| = 0.
\end{equation}
\label{for function}
\end{lemma}

\begin{proof}

We consider the projections $P_{[1,n]}^s:X_s\rightarrow\left(\sum_{k=1}^{n}\oplus(X_{s_k}\oplus\ell_{\infty}(\Delta_k))_{\infty}\right)_\infty$ where $\cup_n\Delta_n=\Gamma(\T_s,L_s,\delta_s)$. By Proposition \ref{scalar plus hc} there exists a real number $\lambda_s$ such that $K_s=T_s - \lambda_s I_s$ is horizontally compact. We will show that $\la_s$ is the desired scalar.

To find a sequence of compact operators $(C_n)_n$ so that \eqref{i am fifteen euros short} is satisfied, it is evidently enough to that for every $\delta > 0$ there is $n_0\in\N$ so that for every $n\geq n_0$ there is a compact operator $C$ so that
\begin{equation}\label{i wonder whether it will stop raining soon enough}
\left\|T_s -\left(\la_s I_s + \sum_{i=1}^n\left(I_{s_i}\circ (T_s - \la_s I_s) \circ I_{s_i}\right)\right) - C\right\| < \delta.
\end{equation}
Fix $\delta>0$ and let $n_0 \in \mathbb{N}$ such that $\|K_s - K_s \circ P_{[1,n]}^s\|< \delta$ for all $n\geq n_0$. Observe that the operator $C' = K_s\circ P_{[1,n_0]}^s-\sum_{n=1}^{n_0}K_s\circ I_{s_n}$ is compact and by Lemma \ref{significant part of T} we have that the operator $\tilde{C} = \sum_{n=1}^{n_0}K_s\circ I_{s_n}-\sum_{n=1}^{n_0}I_{s_n}\circ K_s\circ I_{s_n}$ is compact as well. Setting $C = C' + \tilde{C}$ it is easy to check that $C$ is the desired operator.

In order to show that $\lambda_s$ is unique, let $\tilde{\lambda}_s$ be a scalar so that there exists a sequence of compact operators $(\tilde{C}_n)_n$ with
\begin{equation}\label{it is ok i will use your card}
\lim_n\left\|T_s -\left(\tilde{\lambda}_s I_s + \sum_{i=1}^n\left(I_{s_i}\circ (T_s - \tilde{\lambda}_s I_s) \circ I_{s_i}\right)\right) - \tilde{C}_n\right\| = 0.
\end{equation}
Assume that $\tilde{\lambda}_s \neq \lambda_s$ and choose a sequence of compact operators $(C_n)_n$ so that \eqref{i am fifteen euros short} is satisfied. Combining \eqref{i am fifteen euros short} and \eqref {it is ok i will use your card} we conclude:
\begin{equation*}
\lim_n\left\|I_s - \left(\sum_{i=1}^nI_{s_i} + \frac{1}{\tilde{\lambda}_s - \lambda_s}(C_n - \tilde{C}_n)\right)\right\| = 0.
\end{equation*}
This implies that the identity operator on $X_s$ is horizontally compact, which is absurd.


In order to prove that $|\la_s|\leq \|T_s\|$ fix $\delta>0$ and choose $n_0\in\N$ and a compact operator $C$ so that \eqref{i wonder whether it will stop raining soon enough} is satisfied for $n=n_0$. As $C$ is compact, we may choose $x\in P_{(n_0,\infty)}^s(X_s)$ with $\|x\|=1$ such that $\|Cx\|<\delta$. Considering the above we have that \[|\lambda_s|=\|\lambda_sx\|\leq \|T_s(x)-\la_s(x)-C(x)\|+\|T_s(x)\|+\|C(x)\|\leq 2\delta+\|T_s\|.\]
Since $\delta$ was chosen arbitrarily the proof is complete.
\end{proof}

\begin{corollary}
Let $\T,L,\e$ be as in Definition \ref{def xtle}. Then every strictly singular operator $T:\X\rightarrow\X$ is compact.\label{ss implies compact}
\end{corollary}

\begin{proof}
We shall use induction on the rank $o(\T)$ of the tree $\T$. The case where $o(\T)=0$ (i.e. $\X=X_{AH(L)}$ for some $L\subseteq\N$ infinite) follows by the Argyros-Haydon method of construction in \cite{AH}.

Suppose now that $o(\T)=\alpha$ and that the thesis is true for every $\s,M,\delta$ such that $o(\s)<\alpha$. Let $T:\X\rightarrow\X$ be a strictly singular operator and let also $\{s_n: n\inn\}$, $(L_n)_{n=0}^\infty$, $\delta>0$
 as in Definition \ref{def xtle} (2) such that $X_{(\T,L,\e)} = \left(\sum_{n=1}^{\infty}\oplus X_{s_n} \right)_{AH(L_0)}$ where $X_{s_n}=X_{(\T_{s_n},L_n,\de)}$.

Since $T$ is strictly singular it is not hard to see that $T$ is horizontally compact. Indeed, it follows that for every closed subspace generated by a bounded horizontally block sequence $(x_n)_n$ there exists a further block subspace $Y$ generated by a block sequence $(y_n)_n$ of $(x_n)_n$ such that the operator $T|_Y$ is compact and thus horizontally compact. By Proposition \ref{scalar plus hc} let $\lambda$ be a scalar such that $T-\lambda I$ is horizontally compact. It follows that $\lambda=0$.

Lemma \ref{for function} yields the following \begin{equation*}
\lim_n\left\|T -\sum_{i=1}^n\left(I_{s_i}\circ T \circ I_{s_i}\right) - C_n\right\| = 0.
\end{equation*}
Since $o(\T_{s_i})<\alpha$ the inductive assumption applied on $\T_{s_i},L_{s_i},\de$ yields that for each $i$ the strictly singular operator $I_{s_i}\circ T \circ I_{s_i}:X_{s_i}\to X_{s_i}$ is compact and by the above the result follows.
\end{proof}

Recall that a tree $\T$ becomes a Hausdorff compact topological
space, if it is equipped with the topology having the sets
$\T_t$, $t\in\T$ as a subbase. We are now finally ready to state the main result of this section, which states that every operator defined on the space $\X$ is approximated by a sequence of operators, each one of which is linear combination of the projections $I_s, s\in\T$ plus a compact operator.

\begin{theorem}\label{function}
Let $\T,L,\e$ be as in Definition \ref{def xtle} and
$T:\X\rightarrow\X$ be a bounded linear operator. Then there
exists a unique function $f:\T\rightarrow\mathbb{R}$ such
that $\|f\|_\infty \leq
\|T\|$ and it satisfies the following:
if we set $\mu_{\es_{\T}} = f(\es_{\T})$ and for every node
$s\neq\es_{\T}$ we set $\mu_s = f(s) - f(s^-)$, then there exists
an increasing sequence $(\s_n)_n$ of finite downwards closed subtrees of $\T$ with
$\T = \cup_n\s_n$ and a sequence of compact operators $(C_n)_n$
such that
\begin{equation}\label{mousakas}
\lim_n \left\|T - \sum_{s\in\s_n}\mu_sI_s - C_n\right\| = 0.
\end{equation}
Moreover, the function $f$ is continuous.
\end{theorem}

\begin{proof}
We define the function $f$ as follows, for every non-maximal node $s$, set $f(s) = \la_s$ to be the real number provided by Lemma \ref{for function} and for every maximal node $s$, $f(s)=\la_s$ to be the unique real number such that $I_s\circ T\circ I_s - \la_s I_s$ is a compact operator. By the definition of the function $f$ and since $\|I_{s}\|=1$ (see Proposition \ref{tree xtle}) it immediately follows that $\|f\|_{\infty}\leq \|T\|$.

For the rest of the proof we use induction on $o(\T)$. If $o(\T)=0$, then as stated above $f(\es_{\T})$ is the unique real number $\lambda$ such that $K=T-\lambda I_{\es_{\T}}$ is compact and \eqref{mousakas} holds.

Assume that $o(\T)=\alpha$ and that the statement is true for every tree $\T',L',\e'$ as in Definition \ref{def xtle} with $o(\T')<\alpha$. Let also $\{t_n: n\inn\}$, $(L_n)_{n=0}^\infty$, $\delta>0$
 as in Definition \ref{def xtle} (2) such that $X_{(\T,L,\e)} = \left(\sum_{n=1}^{\infty}\oplus X_{t_n} \right)_{AH(L_0)}$ where $X_{t_n}=X_{(\T_{t_n},L_n,\de)}$.

Since $o(\T_{t_n})<\alpha$ we apply the inductive assumption to each $\T_{t_n},L_n,\de$, $K_{t_n}=I_{t_n}\circ (T-\mu_{\es_T} I)\circ I_{t_n}$ and we obtain a unique continuous function $f_n:\T_{t_n}\to\mathbb{R}$ with $\|f_n\|_{\infty}\leq\|K_{t_n}\|$, an increasing sequence $(\s^{t_n}_i)_i$ of finite downwards closed subtrees of $\T_{t_n}$ with
$\T_{t_n} = \cup_i\s^{t_n}_i$ and a sequence of compact operators $(C^{t_n}_i)_i$
such that if $\tilde{\mu}_{\es_{\T_{t_n}}}=f_n(t_n)$ and $\tilde{\mu}_s=f_n(s)-f_n(s^-)$ for every $s\in\T_{t_n}$, $s\neq\es_{\T_{t_n}}=t_n$, then the following holds:
\begin{equation*}
\lim_i \left\|T_{t_n} - \sum_{s\in\s^{t_n}_i}\tilde{\mu}_sI_s - C^{t_n}_i\right\| = 0.
\end{equation*}
Observe that by Lemma \ref{for function} and the definition of the functions $f$, $f_n$ it follows that $f(s)-\mu_{\es_T}=f_n(s)$ and therefore $\mu_s=\tilde{\mu}_s$ for every $s\in\T_{t_n}$. Since $f(\es_{\T})=\mu_{\es_{\T}}$ such that $T-\mu_{\es_{T}}I$ is horizontally compact, the uniqueness of $f_n$ implies that $f$ is unique. Moreover by Lemma \ref{for function} there exists a sequence of compact operators $(C_n^\prime)_n$ such that
\begin{equation*}
\lim_n\left\|T -\left(\mu_{\es_{\T}}I_{\es_{\T}}+\sum_{i=1}^nK_{t_i}\right) - C'_n\right\| = 0
\end{equation*}

Note that $\T=\cup_n\T_{t_n}$ and for each $n$ we set $\s_n=\cup_{i=1}^n\s^{t_i}_n$, $C_n=\sum_{i=1}^nC^{t_n}_i+C'_n$. It follows that $\T=\cup_n\s_n$, where $(\s_n)_n$ is an increasing sequence of finite downwards closed subtrees, $(C_n)_n$ is a sequence of compact operators defined on $\X$ and using a diagonalization argument we conclude that
\begin{equation*}
\lim_n \left\|T - \sum_{s\in\s_n}\mu_sI_s - C_n\right\| = 0.
\end{equation*}

It remains to show that $f$ is continuous. Observe that it is enough to show that $f$ is continuous on $\es_{\T}$ or equivalently that $(f(s_n))_n$ converges to $f(\es_{\T})$. We recall by the above that $\|f_n\|_{\infty}\leq \|I_{t_n}\circ (T-\mu_{\es_{\T}})\circ I_{t_n}\|$. Since $T-\mu_{\es_{T}}$ is horizontally compact it follows that $(f_n)_n$ converges norm to 0. Since $|f_n(s_n)|\leq\|f_n\|_{\infty}$ and $f_n(s_n)=f(s_n)-\mu_{\es_{\T}}$ the proof is complete.
\end{proof}

\section{The Calkin algebras of the spaces $\X$}\label{the calkins}
As it was proved in the previous section,  every bounded linear operator defined on $\X$ is approximated by a sequence of operators, each one of which is a linear combination of the projections $I_s, s\in\T$ plus a compact operator. For a given operator, these linear combinations define a continuous function with domain the tree $\T$, which is used to define a map $\F:\C al(\X)\rightarrow C(\T)$ which is an onto bounded algebra homomorphism.

\begin{remark}\label{operators separable}
By Remark \ref{scriptLinfty} and  \cite[Theorem 5.1]{JRZ} we conclude that the space $\X$ has a basis. Also, by Proposition \ref{l1 predual}, the dual $\X^*$ is separable and hence the space of all compact operators on $\X$ is separable as well. Theorem \ref{function} clearly yields that the space $\langle\{I_s:\;\s\in\T\}\rangle+\mathcal{K}(\X)$ is dense in $\mathcal{L}(\X)$ and hence the space of all bounded linear operators on $\X$ is separable. Therefore $\C al(\X)$, the Calkin algebra of $\X$ is separable, in particular the linear span of the set $\{[I_s]:\;s\in\T\}$ is dense in $\C al(\X)$.
\end{remark}

\begin{proposition}
Let $\T, L, \e$ be as in Definition \ref{def xtle}. We define a map $\Ft:
\mathcal{L}(\X)\rightarrow C(\T)$ such that for
every operator $T$, $\Ft(T)$ is the function provided by Theorem
\ref{function}. Then $\Ft$ is a norm-one algebra homomorphism with dense
range and $\ker\Ft = \mathcal{K}(\X)$.\label{map fi tilde}
\end{proposition}

\begin{proof}
The fact that $\Ft$ has norm at most one follows from Theorem \ref{function}, in particular the fact that $\left\|\Ft(T)\right\|_\infty \leq \|T\|$ for every bounded operator $T$. Also $\Ft$ maps the identity map to the constant unit function and therefore $\Ft$ has norm one.

We now show that $\Ft$ has dense range and that it is an algebra homomorphism on the space $\langle\{I_s: s\in\T\}\rangle$. First, by Proposition \ref{tree xtle} (ii) observe that for each $s\in\T$ the image $\Ft(I_s)$ coincides with the characteristic function upon the  subtree $\T_s$, denoted as $\mathcal{X}_{\T_s}$.
From this it also follows that the image of $\Ft$ is dense in $C(\T)$. Moreover, observe that for $S=\sum_{i=1}^n\lambda_i I_{s_i}$, $T=\sum_{i=1}^m\mu_i I_{t_i}$ we have that $\Ft(T\circ S)=\Ft(T)\cdot \Ft(S)$.


We now show that $\ker\Ft = \mathcal{K}(\X)$. First we observe that if $\Ft(T)=0$ then by Lemma \ref{function} $T$ is the limit of compact operator and thus $T$ is compact. Now let $T$ be a compact operator. Observe that the zero function satisfies the conclusion of Theorem \ref{function} and by uniqueness we conclude that $T$ is in $\ker\Ft$.

Since, by Remark \ref{operators separable}, the space $\langle\{I_s: s\in\T\}\rangle + \mathcal{K}(\X)$ is dense in $\mathcal{L}(\X)$, it also follows that $\Ft$ is an algebra homomorphism.
\end{proof}

\begin{remark}
Let $\T, L, \e$ be as in Definition \ref{def xtle}. Then by the above it follows that the operator,
$\F:\C al(\X)\rightarrow C(\T)$, defined by the rule $\F([T]) =
\Ft(T)$, is a 1-1 algebra homomorphism with dense range and $\|\F\| =
1$.\label{map fi}
\end{remark}

\begin{proposition}
Let $\T,L,\e$ be as in Definition \ref{def xtle}. Then $\F$ is a bijection, i.e. $\mathcal{C}al(\X)$ is isomorphic, as a Banach algebra, to $C(\T)$. More precisely, we have that $\|\F\|\left\|\F^{-1}\right\| \leq 1 + \e$.\label{finite rank onto}
\end{proposition}

\begin{proof}
As $\F$ is a norm one algebra homomorphism with dense range it is enough to show that it is bounded below and it is evidently enough to do so for
a dense subset of $\C al(\X)$. We will show that for every bounded linear operator $T$ on $\X$, that is a finite linear combination of
the $I_s$, $s\in\T$, there is a compact operator $K$ on $\X$ so that
\begin{equation}
\|T - K\| \leq
(1+\e)\|\F([T])\|\label{finite
rank onto1}
\end{equation}
which of course yields that $\left\|[T]\right\| \leq
(1+\e)\|\F([T])\|$.

Let us first first make a few simple observations:
\begin{itemize}

\item[(i)] For every $s\in\T$ we have that $\F([I_s])$ is equal to $\mathcal{X}_{\T_s}$, the characteristic function of the clopen set $\T_s$.

\item[(ii)] For every real numbers $(\la_s)_{s\in\T}$, finitely many of which are not zero, if $f = \sum_{s\in\T}\la_s\mathcal{X}_{\T_s}$  we have that
\begin{equation*}
\|f\| = \max_{s\in\T}\left|\sum_{\es_{\T}\leq t\leq s}\la_t\right|.
\end{equation*}


\end{itemize}
The first observation follows trivially from the definition of the
map $\F$ while the second one is an immediate consequence of the
fact that $f(s) = \sum_{\es_{\T}\leq t\leq s}\la_t$ for every
$s\in\T$.

By observations (i) and (ii) it is enough to show that if $T = \sum_{s\in\T}\lambda_s I_s$, then there is a compact operator $K$ so that $\|T - K\| \leq (1+\e) \max_{s\in\T}\left|\sum_{\es_{\T}\leq t\leq s}\la_t\right|.$


We are now ready to prove \eqref{finite rank onto1}, by induction on $o(\T)$. If $o(\T)=0$, then $\T = \{\es_{\T}\}$ and $\F$ is an isometry onto the one-dimensional Banach space $C(\T)$. Assume that $o(\T)=\alpha$ and that the statement is true for every tree $\T',L',\e'$ as in Definition \ref{def xtle} with $o(\T')<\alpha$. Let also $\{t_k: k\inn\}$, $(L_k)_{k=0}^\infty$, $\delta>0$ be as in Definition \ref{def xtle} (2) such that $X_{(\T,L,\e)} = \left(\sum_{k=1}^{\infty}\oplus X_{t_k} \right)_{AH(L_0)}$ where $X_{t_k}=X_{(\T_{t_k},L_k,\de)}$.

 Let $[T] =
\sum_{s\in\T}\la_s[I_s]$ be a finite linear combination of the
$[I_s], s\in\T$. Choose $n\in\N$ so that $\la_s = 0$ for every $s\notin \{\es_{\T}\}\cup\left(\cup_{k=1}^n\T_{t_k}\right)$. For $k=1,\ldots,n$ define the operators $T_k' = \sum_{s\in\T_{t_k}}\la_sI_s$ and $T_k = \la_{\es_{\T}}I_{t_k} + T'_k$. Observe that
\begin{equation}\label{we need this for the induction}
T_k = (\la_{\es_{\T}} + \la_{t_k})I_{t_k} + \sum_{s>t_k}\la_s I_s.
\end{equation}
and rewrite $T$ as follows:
\begin{eqnarray}
T &=& \la_{\es_{\T}}P_{[1,n]} + \sum_{k=1}^nT_k + \la_{\es_{\T}}P_{(n,+\infty)}\nonumber\\
 &=& \la_{\es_{\T}}\sum_{k=1}^n\left(P_{k} - I_{t_k}\right) + \sum_{k=1}^n\left(\la_{\es_{\T}}I_{t_k} + T_k\right) + \la_{\es_{\T}}P_{(n,+\infty)}\nonumber\\
 &=& \sum_{k=1}^nT_k + \la_{\es_{\T}}P_{(n,+\infty)} + \la_{\es_{\T}}\sum_{k=1}^n\left(P_{k} - I_{t_k}\right)\label{getting there}
\end{eqnarray}
For $k=1,\ldots,n$ we write $T_k = \tilde{S}_k$ where $S_k$ is an operator  on $X_{t_k}$ (see Remark \ref{sums of ops}). The inductive assumption and \eqref{we need this for the induction} yield that there is a compact operator $C_k$ on $X_{t_k}$ so that
\begin{equation}
\|T_k - \tilde{C}_k\| \leq (1 + \delta)\max_{s\in \T_{t_k}}\left|\la_{\es_{\T}} + \sum_{t_k\leq t\leq s}\la_t\right|
\end{equation}
Using the above and applying Proposition \ref{all the money} there is a compact operator $K'$ on $\X$ so that if
\begin{equation*}
S = \sum_{k=1}^n(T_k - \tilde{C}_k) + \la_{\es_{\T}}P_{(n,+\infty)} - K'
\end{equation*}
then
\begin{eqnarray}
\|S\| &\leq&
(1+\de)\max\left\{\max_{1\leq k\leq n}(1+\de)\max_{s\in \T_{t_k}}\left|\la_{\es_{\T}} + \sum_{t_k\leq t\leq s}\la_t\right|, |\la_{\es_{\T}}|\right\}\nonumber\\
&\leq& (1+\de)^2\max_{s\in\T}\left|\sum_{\es_{\T}\leq t\leq s}\la_t\right|.\label{upside down}
\end{eqnarray}
Finally, set
\begin{equation*}
K = \sum_{k=1}^n\tilde{C}_k + K' - \la_{\es_{\T}}\sum_{k=1}^n(P_{k} - I_{t_k})
\end{equation*}
By \eqref{getting there} we have that $T - K = S$. By Remark \ref{difference compact}, the choice of $\de$ and \eqref{upside down} we conclude that $K$ is the desired compact operator.
\end{proof}

We conclude this section with some remarks concerning the space of bounded operators on $\X$.

\begin{remark}\label{the swan remark}
The space $\mathcal{L}(\X)$ does not contain an isomorphic copy of $c_0$. Indeed, towards a contradiction assume that there is a sequence of operators $(T_k)_k$ on $\X$ equivalent to the unit vector basis of $c_0$. It follows that for every $x\in\X$ and $x^*\in\X^*$ the series $\sum_kx^*T_kx$ converges absolutely. By Remark \ref{hi saturated} $c_0$ does not embed into $\X$. A well know theorem by Bessaga and Pe\l czyn\' ski yields that for every $x\in\X$ the series $\sum_kT_kx$ converges unconditionally which implies that the operator $R:\ell_\infty(\N)\rightarrow\mathcal{L}(\X)$ with $R(a_k)_k = SOT-\sum_ka_kT_k$ is well defined and bounded. By \cite[Proposition 1.2]{R} there is an infinite subset $L$ of $\N$ so that $R$ restricted onto $\ell_\infty(L)$ is an isomorphic embedding. Remark \ref{operators separable} yields a contradiction.
\end{remark}

\begin{remark}\label{the black swan remark}
The quotient map $Q:\mathcal{L}(\X)\rightarrow\mathcal{C}al(\X)$ is strictly singular. Indeed, by Proposition \ref{finite rank onto} $\mathcal{C}al(\X)$ is isomorphic to $C(\T)$. If $o(\T) = 0$ then $\mathcal{C}al(\X)$ is one-dimensional and the result trivially holds. Otherwise, $\T$ is an infinitely countable compact metric space, i.e. $C(\T)$ is $c_0$ saturated and hence so is $\mathcal{C}al(\X)$. Remark \ref{the swan remark} yields that the quotient map $Q$ is strictly singular.
\end{remark}

\begin{remark}\label{the fuchsia swan remark}
In \cite{T} a space $\mathfrak{X}_\infty$ is presented whose Calkin algebra is $\ell_1$. It follows that the space of compact operators on $\mathfrak{X}_\infty$ is complemented in the space of bounded operators. This is also the case for $\X$ if $o(\T) = 0$, as the compact operators are of co-dimension one in the space of bounded operators. However, if $o(\T)> 0$ this is is no longer the case, i.e. $\mathcal{K}(\X)$ is not complemented in $\mathcal{L}(\X)$. Indeed, if we assume that there is a subspace $Y$ of $\mathcal{L}(\X)$ so that $\mathcal{L}(\X) = \mathcal{K}(\X)\oplus Y$, the open mapping theorem implies that $Q|_Y:Y\rightarrow \mathcal{C}al(\X)$ is an onto isomorphism. Since $o(\T)>0$ we conclude that $Y$ is necessarily infinite dimensional, which contradicts Remark \ref{the black swan remark}.
\end{remark}

\begin{remark}\label{the salmon swan remark}
By Remark \ref{the swan remark}, $\mathcal{K}(\X)$ does not contain $c_0$. If moreover $o(\T) > 0$, then by Remark \ref{the fuchsia swan remark} $\mathcal{K}(\X)$ is not complemented in $\mathcal{L}(\X)$. This is related to Question B from \cite{E} and it is, to our knowledge, the first known example of a Banach space where the space of compact operators does not contain $c_0$ and is at the same time not complemented in the space of bounded operators.
\end{remark}

\begin{remark}\label{the prussian blue swan remark}
Corollary \ref{ss implies compact} and Remark \ref{the black swan remark} imply that if $o(\T) > 1$, then for every $\de>0$ there is non-strictly singular operator defined on $\X$ that is $\de$-close to a compact one. For example, if $\{s_n:\;n\in\N\}$ are the immediate successors of the root of $\T$, then Remark \ref{the swan remark} implies that there is a finite subset $F$ of $\N$ so that $\left\|\sum_{k\in F}I_{s_k}\right\| \geq 2/\de$. Proposition \ref{all the money} yields that $T = \left\|\sum_{k\in F}I_{s_k}\right\|^{-1}\left(\sum_{k\in F}I_{s_k}\right)$ is such an operator.
\end{remark}

\section{Some consequence}\label{main stuff}

In this final section we conclude that for every countable compact metric space $K$, the algebra $C(K)$ is homomorphic to the Calkin algebra of some Banach space.

\begin{theorem}\label{i am telling you now that this is the final theorem of this paper of ours}
Let $K$ be a countable compact metric space. Then there exists a $\mathcal{L}_\infty$-space $X$, with $X^*$ isomorphic to $\ell_1$, and a norm-one algebra homomorphism $\Phi:\C al(X)\rightarrow C(K)$ that is one-to-one and onto. Even more, for every $\e>0$ the space $X$ can be chosen so that $\left\|\Phi\right\|\left\|\Phi^{-1}\right\| \leq 1+\e$.\label{finite cb index}
\end{theorem}

\begin{proof}
A well known Theorem by Sierpinski and Mazurkiewicz implies that there exist an ordinal number $\alpha$ and a natural number $n$ so that  $K$ is homeomorphic to the ordinal  number $\omega^\alpha n$. Note that $\omega^\alpha$ is homeomorphic to a tree $\T$ (which is well founded, has a unique root and every non-maximal node of it has countable infinitely many immediate successors) and that this tree is of order $\alpha$.

Let $L_1,\ldots, L_{n}$ be pairwise disjoint infinite subsets of the natural numbers, $\e > 0$ and consider the spaces $X_{(\T,L_1,\e)},\ldots,X_{(\T,L_n,\e)}$. We define the space $X = \left(\sum_{i=1}^{n}\oplus X_{(\T,L_i,\e)}\right)_\infty$ and claim that it has the desired properties. By Proposition \ref{l1 predual} and Remark \ref{scriptLinfty} it easily follows that $X$ is a $\mathcal{L}_\infty$-space with  $X^*$ isomorphic to $\ell_1$. Also, the fact that the sets $L_1,\ldots, L_n$ are pairwise disjoint and Proposition \ref{disjoint weights compact} easily yield that $\mathcal{C}al(X)$ is isometric, as a Banach Algebra, to $\left(\sum_{i=1}^n\oplus\mathcal{C}al(X_{(\T,L_i,\e)})\right)_\infty$, which, by Proposition \ref{finite rank onto}, is $(1+\e)$-isomorphic as a Banach algebra to $\left(\sum_{i=1}^n\oplus C(\omega^\alpha)\right)_\infty$, which is of course isometric as Banach algebra to $C(\omega^\alpha n)$ and this yields the desired result.
\end{proof}

Discovering the variety of Banach algebras that can occur as Calkin algebras is a topic that we believe should be investigated further. We point out that all known examples of Calkin algebra of Banach spaces are either finite dimensional or non-reflexive.
\begin{qst}
Does there exists a Banach space whose Calkin algebra is reflexive and infinite dimensional?
\end{qst}
It is worth mentioning that the method used in this paper does not seem to be able to provide an example of a Banach space whose Calkin algebra is a $C(K)$ space for $K$ uncountable.
\begin{qst}
Does there exists a Banach space whose Calkin algebra is isomorphic, as a Banach algebra, to $C(K)$ for an uncountable compact space $K$?
\end{qst}

\subsection*{Acknowledgement} We wish to thank Professor S. A. Argyros for his special advice and for introducing the problem. The paper was mostly written during  the second author's visit at the National Technical University of Athens. He would like to thank S. A. Argyros and P. Motakis for their kind hospitality. Finally, we would like to thank the referee whose suggestions helped immensely improve this paper. In particular, in an earlier version of this paper we were able to prove Theorem \ref{b} only in the case of countable compact spaces with finite Cantor-Bendixson index. The referee's approach to calculating the norm of operators of the form $\sum_{k=1}^n\oplus T_k$ in the Calkin algebra of $(\sum\oplus X_k)_{AH}$ helped us extend this result to all countable compact metric spaces.

\end{document}